\documentclass[12pt,reqno]{amsart}
\usepackage{graphicx,indentfirst}
\usepackage{amsmath,amssymb,mathrsfs}
\usepackage{amsthm,amscd}
\usepackage{verbatim}
\usepackage{appendix}
\usepackage{enumitem,titletoc}
\usepackage[utf8]{inputenc}
\usepackage{imakeidx}
\makeindex[columns=2, title=Alphabetical Index]
\usepackage{fancyhdr}
\usepackage{amsfonts,color}
\usepackage[all]{xy}
\usepackage{tikz-cd}
\usepackage{syntonly}
\usepackage{fancyhdr}
\usepackage{comment}
\usepackage{array}
\usepackage[left=2.5cm,right=2.5cm,bottom=3cm,top=3cm]{geometry}

\usepackage{tikz}
\usepackage{extarrows}
\usepackage{hyperref}

\def\XXint#1#2#3{{\setbox0=\hbox{$#1{#2#3}{\int}$ }
		\vcenter{\hbox{$#2#3$ }}\kern-.6\wd0}}

\newtheorem{claim}{Claim}
\newtheorem{lem}{Lemma}[section]
\newtheorem{prop}{Proposition}[section]
\newtheorem{ques}{Question}

\newtheorem{defn}{Definition}[section]
\newtheorem{corr}{Corollary}[section]

\newtheorem{remark}{Remark}[section]
\newtheorem{theorem}{Theorem}[section]

\allowdisplaybreaks
\newcommand{\del}{\partial}

\newcommand{\eps}{\varepsilon}

\makeindex

\hypersetup{
	colorlinks=true,
	citecolor=green,
	filecolor=green,
	linkcolor=blue,
	urlcolor=black
}

\numberwithin{equation}{section}

\title{Boundary regularity of optimal transport maps on convex domains}
\author{Tristan C. Collins}
\email{tristanc@math.toronto.edu}
\author{Freid Tong}
\email{freid.tong@utoronto.ca}
\address{Department of Mathematics, University of Toronto, 40 St. George St., Toronto, Ontario, Canada
}
\date{\today}

\begin{document}

\maketitle
\begin{abstract}
We study the regularity of optimal transport maps between convex domains with quadratic cost. For nondegenerate Dini continuous densities, we prove  $W^{2,p}$-regularity of the potentials up to the boundary for all $p>0$. If in addition the boundary is $C^{1, \alpha}$ and the densities are $C^{\alpha}$ we improve this to $C^{2, \alpha}$-regularity. We also investigate pointwise $C^{1, 1}$-regularity at boundary points. We obtain a complete characterization of pointwise $C^{1, 1}$-regularity for planar polytopes in terms of the geometry of tangent cones. Furthermore, we study the regularity of optimal transport maps with degenerate densities on cones, which arise from recent developments in K\"ahler geometry. The main new technical tool we introduce is a monotonicity formula for optimal transport maps on convex domains which characterizes the homogeneity of blow-ups. 
\end{abstract}
\tableofcontents

\section{Introduction}

In this paper we study the global regularity properties of solutions of the Monge-Amp\`ere equation
\begin{equation}\label{eq: OTIntro}
\begin{aligned}
\det D^2u &= \frac{g(x)}{g'(\nabla u(x))} \qquad x \in \Omega\\
\nabla u(\Omega) &= \Omega'
\end{aligned}
\end{equation}
where $\Omega, \Omega'$ are convex domains.  This type of equation arises naturally in the study of optimal transportation with quadratic cost, as well as in various aspects of pure and applied mathematics \cite{DPF, Evans, RR, Villani, Villani2}. The existence and uniqueness of a weak solution to~\eqref{eq: OTIntro} was established by Brenier \cite{Brenier}.  Regularity of solutions to ~\eqref{eq: OTIntro} has been a central issue in the field.  In a series of landmark papers \cite{Caffarelli, Caffarelli2, Caffarelli3}, Caffarelli observed that the convexity of the domains is essential in order to establish any regularity of $u$, and he proved the following regularity results:
\begin{itemize}
\item If $\Omega, \Omega'$ are bounded convex domains, and $C^{-1} < g,g'< C$, then $u$ is globally $C^{1,\delta}$ continuous for some $\delta\in (0, 1)$ \cite{Caffarelli2}.  This conclusion continues to hold if $g, g'$ are allowed to degenerate, but satisfy a doubling type condition \cite{Jhaveri-Savin}.
\item If $\Omega, \Omega'$ are uniformly convex with $C^2$ boundary, $C^{-1}<g,g'<C$, and $g,g'$ are $C^{\alpha}$-H\"older continuous, then $u\in C^{2,\alpha'}(\overline{\Omega})$ \cite{Caffarelli3}. Earlier works of Delanoe \cite{Delanoe} (in dimension $2$), and Urbas \cite{Urbas} established the global $C^{2,\alpha'}$ regularity assuming $\del\Omega, \del\Omega'$ are $C^{2,1}$, and $g,g' \in C^{1,1}$.
\end{itemize}
An important recent result of Chen-Liu-Wang \cite{Chen-Liu-Wang} removed the uniform convexity condition, proving that $u \in C^{2,\alpha}(\overline{\Omega})$ provided $C^{-1}<g,g'<C$ are $C^{\alpha}$, and $\Omega, \Omega'$ are convex and $C^{1,1}$-regular. This result is rather surprising, since the strict convexity of the boundary is a necessary condition for global $C^{2,\alpha}$ regularity of solutions to the Dirichlet problem for the Monge-Amp\`ere equation \cite{CNS, Cheng-Yau, Savin2, Trudinger-Wang}. Assuming $\Omega, \Omega'$ are convex and $C^{1,1}$ regular,  Chen-Liu-Wang also establish the global $C^{1,1-\eps}$ regularity (for any $\eps>0$), and $W^{2,p}$ regularity (for any $p>0$) assuming that $C^{-1}<g,g'<C$ are continuous up to the boundary.  Previously,  Caffarelli established interior $W^{2,p}$ estimates in \cite{Caffarelli4}, while global $W^{2,p}$ estimates were obtained for the Dirichlet problem by Savin \cite{Savin}, and for the second boundary value problem by Chen-Figalli \cite{Chen-Figalli}.  

Much less is known about the boundary regularity of $u$, beyond Caffarelli's $C^{1, \delta}$-estimate, if one does not impose any regularity assumptions on the boundary. In this direction, a notable result of Savin-Yu \cite{Savin-Yu} in dimension $2$ establishes the global $C^{1,1-\eps}$ and $W^{2,p}$ regularity of solutions when $g=g'=1$ for arbitrary convex domains $\Omega, \Omega' \subset \mathbb R^2$. 

In this paper we develop new techniques for the regularity of optimal transport maps on general convex domains. Our first result is a global $W^{2,p}$ and $C^{1,1-\eps}$ estimate for solutions of ~\eqref{eq: OTIntro} in arbitrary convex domains in any dimension, assuming $C^{-1}\leq g,g'\leq C$ are Dini continuous. This extends the results of Savin-Yu \cite{Savin-Yu} to all dimensions. 
\begin{theorem}\label{thm: INTROC1aW2pregularity}
	   Let $u$ satisfy ~\eqref{eq: OTIntro}, $\Omega, \Omega'$ are convex, and $g(x), g'(y)$ are positive and Dini continuous in $\overline{\Omega}$ and $\overline{\Omega'}$ respectively. Then for any $\eps>0$, $u\in C^{1, 1-\eps}(\overline\Omega)$, and for any $p<\infty$, $u\in W^{2, p}(\Omega)$. 
\end{theorem}
See Theorems~\ref{thm: global-C-1-1-eps-regularity} and ~\ref{thm: global-W-2-p-regularity} for the more precise and quantitative statement.

We also obtain a global $C^{2,\alpha}$-regularity result, which extends the result of Chen-Liu-Wang \cite{Chen-Liu-Wang} from $C^{1,1}$ convex domains to $C^{1,\beta}$ convex domains. 
\begin{theorem}\label{thm: INTRO-C-2-a-regularity}
        Suppose that $u$ satisfies ~\eqref{eq: OTIntro}, $0<g(x)\in C^{\alpha}(\overline{\Omega}), 0<g'(y)\in C^{\alpha}(\overline{\Omega'})$. Suppose in addition $\Omega, \Omega'$ are convex and have $C^{1, \beta}$-boundary for some $\beta\in(0, 1)$. Then $u\in C^{2, \gamma}(\overline\Omega)$ for $\gamma = \min(\alpha, \beta)$.
\end{theorem}
See Theorem~\ref{thm: global-C-2-a-regularity} for the more precise and quantitative statement. 

The exponents in both Theorem~\ref{thm: INTROC1aW2pregularity} and ~\ref{thm: INTRO-C-2-a-regularity} are optimal. In particular, it is known that for general convex domains, solutions of ~\eqref{eq: OTIntro} can fail to be globally $C^{1, 1}$. A concrete example of this failure was first explicitly exhibited in \cite{JMPS}. Given this, we study the pointwise $C^{1, 1}$-regularity of $u$ at boundary points and relate it to the geometry of tangent cones of $\Omega$ and $\Omega'$ at the boundary points; see Theorem~\ref{thm: non-round},~\ref{thm: strong-oblique-implies-round}, ~\ref{thm: 2D-case-4-round}, and ~\ref{thm: round-degenerate-density}. We define a property called {\em roundness}, which roughly says that the solution exhibits homogeneous behaviour at the level of its sub-level sets. If $g, g'$ are non-degenerate, roundness is equivalent to being pointwise $C^{1, 1}$, and thus our results can be used to study when pointwise $C^{1, 1}$-regularity holds. For convex polytopes in $\mathbb{R}^2$, we completely characterize the pointwise $C^{1,1}$-regularity of solutions to~\eqref{eq: OTIntro} when $g=g'=1$; see Theorem~\ref{thm: planar-C-1-1-reg-characterization}. More generally, we study the regularity properties of optimal transport maps with degenerate, homogeneous densities in arbitrary convex domains, motivated in part by recent progress in K\"ahler geometry \cite{Collins-Li, Collins-Tong-Yau, Li-intermediate}; see Theorem~\ref{thm: round-degenerate-density} and~\ref{thm: secMixedHomog}.

The main new technical tool we introduce is a monotonicity formula for optimal transport maps on convex domains, which is constant on homogeneous solutions of the optimal transport problem; see Section~\ref{sec: monotonicity-formula}. This formula is heavily motivated by the ideas of Kim-McCann \cite{Kim-McCann} and Kim-McCann-Warren \cite{Kim-McCann-Warren}, interpreting optimal transport maps as high-codimension, space-like, area-maximizing submanifolds of pseudo-Riemannian manifolds. For recent progress on the regularity theory of optimal transport from this point of view, we refer the reader to the recent work of  Brendle-L\'eger-McCann-Rankin \cite{BLMR}, and Yuan \cite{Yuan}.  Technically, the monotonicity formula leads to short proofs of the uniform density property and the uniform obliqueness property for optimal transport maps, and therefore, yields a somewhat simplified proof of the main results of \cite{Caffarelli3, Chen-Liu-Wang}. 

One of the main applications of the monotonicity formula is to deduce that appropriate rescalings of optimal transport maps converge subsequentially to homogeneous optimal transport maps, which play the role of {\em tangent cones} in the context of geometric PDEs.  This allows us to study the regularity of optimal transport maps through the properties of their ``tangent cones". 

 The outline of the paper is as follows.  Section~\ref{sec: preliminaries} recounts the basic notions of optimal transportation and convex geometry that we will need.  Section~\ref{sec: monotonicity-formula} contains the proof of the key monotonicity formula. Section~\ref{Sec: homogeneity-of-blow-ups} proves that appropriate rescalings of optimal transport maps converge to homogeneous optimal transport maps.  In Section~\ref{sec: C-1-1-eps-regularity} we prove Theorems~\ref{thm: INTROC1aW2pregularity} and~\ref{thm: INTRO-C-2-a-regularity}; see Theorem~\ref{thm: global-C-1-1-eps-regularity}, ~\ref{thm: global-W-2-p-regularity}, and Theorem~\ref{thm: global-C-2-a-regularity} for the more precise statement that we prove.

 In the remainder of the paper, we apply our techniques to study the regularity properties of optimal transport maps with possibly degenerate densities on general convex domains.  In Section~\ref{sec: roundness}, we study the shapes of sections of optimal transport maps between cones equipped with homogeneous densities.  We introduce a notion of ``roundness" for sections, which can be viewed as a generalization of $C^{1,1}$-regularity.  We show that roundness is related to the existence of homogeneous optimal transport maps, and we give conditions on the domains $\Omega, \Omega'$ guaranteeing the roundness of sections.  For planar polygonal domains with uniform density, we completely characterize optimal transport maps with round sections. 

Homogeneous optimal transport maps have recently been of independent interest due to connections to the existence of complete Calabi-Yau metrics \cite{Collins-Li, Collins-Tong-Yau} and a problem of Tian-Yau \cite{Tian-Yau, YauICM}.  By solving a free-boundary Monge-Amp\`ere equation, the authors and Yau \cite{Collins-Tong-Yau} showed the existence of homogeneous solutions of the optimal transport problem when $\Omega$ is a strict cone, $g=1$ and $\Omega'= \{y \in \mathbb{R}^n: y_n>0\}$ is a half-space, equipped with the density $g'(y)=y_n^{k}$.  The first author and Firester \cite{Collins-Firester} extended these techniques to solve a rather general class of free-boundary Monge-Amp\`ere equations, which in particular yields the existence of optimal transport maps in the above setting, but with $g(x)$ a general, non-negative homogeneous density which can degenerate on $\del \Omega$.

Finally, in Section~\ref{sec: flat-side} we show that the presence of a flat-sides in the boundary can lead to solutions exhibiting mixed homogeneity.  These results build in an essential way on the work of Jhaveri-Savin \cite{Jhaveri-Savin} who studied the regularity properties of optimal transport maps with degenerate densities between half-spaces.  

\bigskip

\noindent{\bf Acknowledgements:} We are grateful to S. Brendle, R. J. McCann, N. McCleerey, O. Savin, and H. Yu for helpful conversations. 

Parts of this work was done while we were visiting the Centre de Recherche Mathematiques in Montreal as part of the {\em Geometric Analysis} program in summer of 2024, and while the second author was a postdoc at Harvard's Center of Mathematical Sciences and Applications, and at SLMath (formerly MSRI) as part of the program {\em Special geometric structures and analysis} during fall of 2024. We would like to thank these institutes for their support during various phases of this research. 

T.C.C. is supported in part by NSERC Discovery grant RGPIN-2024-518857, and NSF CAREER grant DMS-1944952. F.T. is supported in part by NSERC Discovery grant RGPIN-2025-06760 and NSF grant DMS-1928930. 

\section{Preliminaries}\label{sec: preliminaries}
Let us describe the optimal transport problem. Given domains $\Omega, \Omega'\subset \mathbb R^n$ equipped with measure $\mu, \nu$ respectively with $\mu(\Omega) = \nu(\Omega')$. An optimal transport map (with quadratic cost) is a measurable map $T:\Omega\to \Omega'$ that satisfy $T_{\sharp}\mu = \nu$ and that minimizes the cost function 
\[c(T) := -\int_{\Omega}x\cdot T(x)d\mu(x)\]
amongst all measurable maps that pushforward $\mu$ to $\nu$. A fundamental theorem of Brenier \cite{Brenier} says that a unique optimal transport map $T$ always exists and it is given by the gradient of a convex potential $T(x) = \nabla u(x)$ for $\mu$-almost every $x$. Moreover, the inverse optimal transport map $T^{-1}:\Omega'\to \Omega$ is given by the gradient of the Legendre transform of $u$, defined by
\[v(y) := \sup_{x\in\Omega}(\langle x, y \rangle-u(x)). \] 
When $\Omega$ and $\Omega'$ are convex, and the measures $d\mu = g(x)dx$, $d\nu = g'(y)dy$ are absolutely continuous with respect to the Lebesgue measure, Caffarelli \cite{Caffarelli} proved that $u$ satisfy the following Monge-Ampere equation in the Alexandrov sense
\begin{equation}\label{eq: OT-MA-equation}
	\begin{cases}
		\det D^2u(x) = \frac{g(x)}{g'(\nabla u(x))} \text{ for }x\in \Omega\\
		\nabla u(\Omega) = \Omega'.
	\end{cases}
\end{equation}
The convexity of the domains plays an important role in Caffarelli's result, as he observed that if the domains are not convex, then Brenier's solutions can fail to be an Alexandrov solution to \eqref{eq: OT-MA-equation}, and $u$ may not even be $C^1$. 

In this paper, we will always assume the domains $(\Omega, \Omega')$ are convex and we will take the approach of using the equation~\eqref{eq: OT-MA-equation} as the definition of an optimal transport map. This has the advantage that it makes sense even if the domain and target are non-compact and have infinite measure, which is necessary to be able to talk about blow-ups. Moreover, we always assume the measures $(g(x)dx, g'(y)dy)$ satisfy a {\em volume doubling condition}, which is the crucial assumption required for Caffarelli's $C^{1, \delta}$-regularity theory to hold \cite{Caffarelli2, Jhaveri-Savin}. For the remainder of this paper, we define an optimal transport map as follows. 

\begin{defn}
    An {\bf optimal transport map} (or more precisely, an optimal transport map between convex domains equipped with doubling measures) will be given by the data $((\Omega, \mu), (\Omega', \nu), u, v)$ where
    \begin{enumerate}
        \item $(\Omega, \Omega')\subset (\mathbb R^n_x, \mathbb R^n_y)$ are two open (not necessarily bounded) convex sets. 
        \item $\mu$ and $\nu$ are measures with support $\overline\Omega$ and $\overline{\Omega'}$ respectively such that
        \begin{enumerate}
            \item They are absolutely continuous with respect to the Lebesgue measure, and we set $d\mu = g(x)dx$, and $d\nu = g'(y)dy$. 
            \item The measures $\mu$ (and $\nu$) are doubling, that is there exist constant $C>1$ such that for any ellipsoid $E\subset \mathbb R^n$ centered at a point $x\in\overline\Omega$ (or centered at $y\in\overline{\Omega'}$), we have
            \[\mu(E)\leq C\mu(\frac{1}{2}E) \left(\text{or  }\nu(E)\leq C\nu(\frac{1}{2}E)\right). \]
            The constant $C>1$ is called the doubling constant of $(\mu, \nu)$. 
        \end{enumerate}
        \item $(u, v)$ are a pair of strictly convex functions on $(\overline\Omega, \overline\Omega')$ such that
        \begin{enumerate}
            \item $\nabla u(\Omega) = \Omega'$ and $\nabla v(\Omega') = \Omega$. 
            \item Let $(\bar u, \bar v)$ be the minimal convex extensions of $(u, v)$,
            \[\bar u(x) := \sup \{l(x): l \text{ is tangent to } u \text{ at some point }x_0\in \Omega\},\]
            \[\bar v(y) := \sup \{l(y): l \text{ is tangent to } v \text{ at some point }y_0\in \Omega'\}. \]
            Then $\bar u, \bar v\in C^{1, \alpha}_{loc}(\mathbb R^n)$ for some $\alpha\in (0, 1)$, and $(\bar u, \bar v) = (v^{\star}, u^{\star})$ where $(v^{\star}, u^{\star})$ are the Legendre transforms of $(v, u)$.   
            \item $(\bar u, \bar v)$ are Alexandrov solutions to the Monge-Amp\'ere equations 
            \begin{equation}
                \det D^2\bar u(x) = \frac{g(x)}{g'(\nabla \bar u(x))}\chi_{\Omega}(x)
            \end{equation}
            and
            \begin{equation}
                \det D^2\bar v(y) = \frac{g'(y)}{g(\nabla \bar v(y))}\chi_{\Omega'}(y). 
            \end{equation}
            respectively.
            \end{enumerate}
    \end{enumerate}
\end{defn}

\begin{remark}
    By \cite[Theorem 1.1]{Jhaveri-Savin} (also see \cite[Remark 2.1]{Savin-Yu}), the properties for $u, v$ are always satisfied by optimal transport maps when the measures $(\mu, \nu)$ are doubling measures supported on convex domains. 
\end{remark}

Given an optimal transport map $((\Omega, d\mu), (\Omega', d\nu), u, v)$, we will always use $x$ to denote the coordinate of the domain and $y$ for the coordinate of the target, which are dual to each other. The following transformations act on the space of optimal transport maps. 
\begin{enumerate}
    \item {\bf Shift}: Given $c\in \mathbb R$, $((\Omega, d\mu), (\Omega', d\nu), u+c, v-c)$ is another optimal transport map. 
    \item {\bf Translation}: Given $x_0\in \mathbb R^n$, there is a shift transformation
    \[\tau_{x_0}((\Omega, d\mu), (\Omega', d\nu), u, v) = ((\Omega+x_0, g(x-x_0)dx), (\Omega', g'(y)dy), u(x-x_0), v(y)+\langle x_0, y\rangle)\]
    and similarly given $y_0\in (\mathbb R^n)'$, there is a shift given by
    \[\tau'_{y_0}((\Omega, d\mu), (\Omega', d\nu), u, v) = ((\Omega, g(x)dx), (\Omega'+y_0, g'(y-y_0)dy), u(x)+\langle x, y_0\rangle, v(y-y_0))\]
    \item {\bf Scaling}: Given $\lambda>0$ there is a scaling transformation
    \[\lambda\cdot ((\Omega, d\mu), (\Omega', d\nu), u, v) = ((\Omega, d\mu), (\lambda^{-1}\Omega', \lambda^{-1}_{\sharp}d\nu), \frac{u(x)}{\lambda}, \frac{v(\lambda y)}{\lambda}).\]
    \item {\bf Affine transformations}: Given a matrix $A\in GL(n, \mathbb R)$, it acts on the space of optimal transport maps by
    \[A\cdot ((\Omega, d\mu), (\Omega', d\nu), u, v) = ((A(\Omega), A_{\sharp}d\mu), (A^{-t}(\Omega'), A^{-t}_{\sharp}d\nu), u\circ A^{-1}, v\circ A^t). \]
\end{enumerate}

\subsection{Basics of Convex geometry}
In this section, we collect some basic concepts of convex geometry that will be used in later sections. 

The following lemma regarding the shape of bounded convex sets is well-known. (see \cite{Guetierrez} for a proof.)
\begin{prop}[John's Lemma]\label{prop: John's Lemma}
    Given any bounded convex set $\Omega$ with non-empty interior and barycenter at the origin. There exist a unique ellipsoid $E$ centered at the origin such that 
    \[E\subset \Omega\subset n^{\frac 3 2}E.\]
    Such a $E$ is called the {\bf John ellipsoid} of $\Omega$. 
\end{prop}
Given two subsets $X, Y\subset \mathbb R^n$, the {\bf Hausdorff distance} between them is defined as 
\[d_H(X, Y) := \sup_{x\in X, y\in Y} \{d(x, Y) , d(y, X)\}\]
The Hausdorff distance induces a topology on the space of compact subsets of $\mathbb R^n$ called the Hausdorff topology. The following proposition is well known. (see \cite{Schneider} for a proof.)
\begin{prop}
    The class of compact convex sets in $\overline{B_R(0)}$ is compact with respect to the Hausdorff topology. Moreover, the volume function is continuous on the class of convex subsets with respect to the Hausdorff topology. 
\end{prop}

\begin{defn}
    We say that a sequence of subsets $\Omega_i\subset\mathbb R^n$ {\bf converges locally in the Hausdorff sense} to a subset $\Omega_{\infty}$ if for any $R>1$, we have \[d_H(\Omega_i\cap \overline{B_R(0)}, \Omega_\infty\cap \overline{B_R(0)})\to 0\] as $i\to \infty$. 
\end{defn}

A {\bf convex cone} $\mathtt C\subset \mathbb R^n$ is a convex set which is closed under scalar multiplication by positive real numbers. Given a convex cone $\mathtt C\subset \mathbb R^n$, there is a {\bf dual cone} $\mathtt C^{\circ}$ defined by
    \[\mathtt C^{\circ} :=\{y\in \mathbb R^n: \langle x, y\rangle\geq 0\text{ for all }x\in\mathtt C\}. \]
If $\mathtt C$ is a closed convex cone, then $\mathtt C^{\circ \circ} = \mathtt C$. 
\begin{defn}
    We say that a convex cone $\mathtt C\subset \mathbb R^n$ is a {\bf strict cone} if it contains no infinite lines, or equivalently if $\mathtt C^{\circ}$ has dimension $n$. 
\end{defn}

We recall the notion of the tangent cone of a convex set at a point on its boundary. 
\begin{defn}
    Let $\Omega$ be an open convex set and $0\in\partial\Omega$. Then the {\bf tangent cone} of $\Omega$ at $0$ is given by the closed convex cone
    \[\mathtt C:= \overline{\{tv: t\geq 0, v\in\Omega\}}. \]
\end{defn}

\subsection{Sections, centered sections, and extrinsic balls}
Let $((\Omega, g(x)dx), (\Omega', g'(y)dy), u, v)$ be an optimal transport map. Following Caffarelli \cite{Caffarelli}, for any $x_0\in \overline\Omega$, $\mathtt{y}_0\in\mathbb{R}^n$ and $h>0$, we define the {\bf section of $u$ at $x_0$ of height $h$ in direction $\mathtt{y}_0$} to be the convex set
\[S_h(u, x_0, \mathtt{y}_0) := \{x\in\Omega: u(x)-u(x_0)-\mathtt{y}_0(x-x_0)\leq h\}.\]
For simplicity, we shall denote
\[
S_{h}(u,x_0):= S_{h}(u,x_0, \nabla u(x_0))
\]
We also define a {\bf centered section} of $u$ at $x_0\in\overline\Omega$ to be
\[S^c_h(u, x_0) := \{x\in \mathbb R^n:\bar u(x)-\bar u(x_0)-p\cdot (x-x_0)\leq h\}\]
where recall that $\bar u$ is the minimal convex extension of $u$, and $p\in \mathbb R^n$ is chosen so that $S^c_h(u, x_0)$ has center of mass $x_0$. By \cite{Caffarelli2}, if the graph of $\bar u$ contains no line, then centered sections always exist. 

Given a centered section $S = S^c_h(u, x_0)$, one can define the normalized section by
\[\tilde S = A(S)\]
where $A$ is the (unique) symmetric positive definite matrix such that $A(E) = B_1(0)$ where $E$ is the John ellipsoid of $S$. We call $A$ the {\bf normalization matrix} of $S$, and it satisfies $(\det A)|S|\sim 1$. We also define the normalized potential $\tilde u:\tilde S\to \mathbb R$ by
\[\tilde u(\tilde x) = \frac{(\bar u-l)(A^{-1}\tilde x)}{h}.\]
where $l(x) = u(x_0)+p\cdot (x-x_0)$. It follows that $\tilde S = \{\tilde u<1\}$, and the pair $(\tilde S, \tilde u)$ is a normalized pair. Moreover, $\nabla \tilde u:\tilde\Omega\to\tilde\Omega'$ where $\tilde\Omega = A(\Omega)$ and $\tilde\Omega' = hA^{-1}(\Omega')$.  

Furthermore, the following properties were established for doubling measures by Caffarelli \cite{Caffarelli2, Caffarelli3} and Jhaveri-Savin \cite{Jhaveri-Savin}.

\begin{prop}\label{prop: properties-of-sections}\cite[Section 3]{Jhaveri-Savin}
There exist constant $c>0$ depends only on dimension, and $C\geq 1$ depending additionally on the doubling constant of $g(x)$ and $g'(y)$ such that
\begin{enumerate}
    \item $B_{C^{-1}}\subset \nabla\tilde u(\tilde S)\subset B_C$
    \[\implies C^{-1}h(S_h^c(u, x_0)-x_0)^{\circ}\subset \nabla (u-l)(S_h^c(u, x_0))\subset Ch(S_h^c(u, x_0)-x_0)^{\circ}.\]
    \item $C^{-1}h^n\leq|S_h^c(u, x_0)||\nabla u(S_h^c(u, x_0))|\leq Ch^n$. 
    \item $S_{ch}^c(u, x_0)\cap \overline\Omega\subset S_h(x_0)\subset S_{Ch}^c(u, x_0)\cap \overline\Omega$. 
    \item $S_{C^{-1}h}(v, \nabla u(x_0))\subset \nabla u(S_h(u, x_0))\subset S_{Ch}(v, \nabla u(x_0))$. 
\end{enumerate}
\end{prop}

In \cite{Caffarelli3, Jhaveri-Savin}, it was also shown that if the measures are doubling, then the centered sections satisfy the following engulfing property. 
\begin{prop}\cite[Lemma 3.3]{Jhaveri-Savin}\label{prop: engulfing}
For any pair of constants $0\leq \underline{t}\leq \overline t\leq 1$, there exists a constant $0<s\leq 1$ such that
\[S_{sh}^c(u, x')\subset \overline tS_h^c(u, x)\]
for all $x\in \overline\Omega$ and $x'\in \underline{t}S_h^c(u, x)\cap \overline\Omega$. Moreover, the constant $s$ depends only on $\underline{t}, \overline t, n$ and the doubling constant of $g$ and $g'$.
\end{prop}

This engulfing property leads to the following $C^{1, \delta}$ estimate \cite{Caffarelli2, Jhaveri-Savin}. (see also \cite{Savin-Yu})    

\begin{prop}\label{prop: C1-delta-estimate}
        Given an optimal transport map $((\Omega, g(x)dx), (\Omega', g'(y)dy), u, v)$ with $0\in\partial\Omega\cap \partial\Omega'$ and $u(0) = |\nabla u|(0) = 0$. Assume that it satisfies
        \begin{enumerate}
            \item $g$ and $g'$ has doubling constant bounded by $K$.
            \item $S_1^c(u, 0)\subset B_K(0)$ and $S_1^c(v, 0)\subset B_K(0)$. 
        \end{enumerate}
        Then there exist $\delta>0$ depending only on $K$ and $n$ such that
        \[\|\bar u\|_{C^{1, \delta}(B_R)}+\|\bar v\|_{C^{1, \delta}(B_R)}\leq C \]
        for $C$ depending only on $n, K, $ and $R$. 
    \end{prop}
    \begin{proof}
        The proof follows from \cite[Proposition 2.5, Remark 2.1]{Savin-Yu} using Proposition~\ref{prop: engulfing}. 
    \end{proof}

Moreover, we will also define the following sets 
\[D_r(u, x_0,y_0) := \{x\in\Omega: (x-x_0)\cdot (\nabla u(x)-y_0)\leq r^2\}\]
which we call {\bf extrinsic balls of radius $r$ centered at $(x_0,y_0)$}.  For simplicity, when $y_0=\nabla u(x_0) $, we shall write $D_r(u, x_0,\nabla u(x_0))=: D_r(u,x_0)$.  These extrinsic balls were also used in \cite{Yuan}. 
The following proposition shows that up to a factor of 2, the extrinsic ball and the sections have the same shape. 
\begin{prop}\label{prop: comparison between sections and extrinsic ball}
Assume that $u$ is strictly convex, and $\underline{x}$ is chosen to that $\nabla u(\underline{x}) = \mathtt{y}_0$ then 
    \[\frac{1}{2}S_{r^2}(u, \underline{x})\subset D_r(u, x_0,\mathtt{y}_0)\subset S_{r^2}(u, x_0, \mathtt{y}_0).\]
    In particular, we have
    \[
    \frac{1}{2}S_{r^2}(u, x_0)\subset D_r(u, x_0)\subset S_{r^2}(u, x_0).
    \]
\end{prop}
\begin{proof}
    Without loss of generality, we can assume $x_0 = 0$, $u(0) =0$, and $\nabla u(0) = 0$ by a shift and subtracting a linear function.  Let $\tilde{u} =u(x)- \langle \mathtt{y}_0,x\rangle$. It follows by convexity that 
    \[
        \tilde{u}(x)=\tilde{u}(x)-\tilde{u}(0) \leq x\cdot \nabla\tilde{u}(x) \leq \tilde{u}(2x)- \tilde{u}(x)
        \] 
        This yields the containment $D_r(u, x_0,\mathtt{y}_0)\subset S_{r^2}(u, x_0, \mathtt{y}_0)$.  Furthermore, if $\mathtt{y}_0=0$, then $\tilde{u}(x)=u(x) \geq 0$ and the second claim follows.  It only remains to prove the containment $\frac{1}{2}S_{r^2}(u, \underline{x})\subset D_r(u, x_0,\mathtt{y}_0)$.  For this, suppose $2x\in S_{r^2}(u,\underline{x})$.  Then we have
        \[
        r^2 \geq \tilde{u}(2x) - \tilde{u}(\underline{x}) \geq \nabla \tilde{u}(x) \cdot x + \tilde{u}(x) - \tilde{u}(\underline{x})
        \]
        On the other hand, by definition of $\underline{x}$ we have $\tilde{u}(x) - \tilde{u}(\underline{x}) \geq 0$, and the result follows.
\end{proof}

\section{Monotonicity formula}\label{sec: monotonicity-formula}
The main goal of this section is to prove a monotonicity formula for optimal transport maps on convex domains.  We note that  Yuan \cite{Yuan} used a somewhat related monotonicity formula in the special case $g=g'=1$ to obtain a new proof of Pogorelov's estimate. 

We state now the exact monotonicity formula.  Later we will prove an effective version of this result, which allows for more general densities; see Proposition~\ref{prop: effective-monotonicity-for Holder-density}.

\begin{theorem}\label{thm: monotonicity}
    Let $((\Omega, g(x)dx), (\Omega', g'(y)dy), u, v)$ be an optimal transport map. Assume $0\in \overline\Omega\cap\overline{\Omega'}$ and $0 = u(0)= |\nabla u|(0)$. Assume in addition that
    \begin{itemize}
        \item[$(i)$]$g\in C^{\alpha}(\overline{\Omega})$, $\Omega \subset \{g>0\}$, and $g$ is homogeneous of degree $l$.
        \item[(ii)]$g'\in C^{\alpha}(\overline{\Omega'})$, $\Omega' \subset \{g'>0\}$, and $g'$ is homogeneous of degree $k$.
    \end{itemize}
    Then the quantity
	\[\chi(r) := r^{-\frac{2(n+l)}{1+\frac{n+l}{n+k}}}\mu(D_r(u, 0)) \]
    is monotone non-increasing in $r$, where \[\mu(D_r(u, 0)) = \int_{\{\psi(x)\leq r^2\}}g(x)dx\]
    for $\psi(x):= x\cdot \nabla u(x)$. 
    More precisely, for any $r_2<r_1$ we have
    \[
    \chi(r_1)-\chi(r_2) \leq  0
    \]
    and $\chi(r)$ is constant on $[0,r_0]$ if and only if $\Omega$ and $\Omega'$ are both conical about the origin and $u$ is homogeneous of degree $1+\frac{n+l}{n+k}$ in $D_{r_0}(u, 0)$. 
\end{theorem}

\begin{remark}
    The precise coercivity term appearing in the monotoncity formula can be deduced from the proof.
\end{remark}

While the basic idea of the proof of Theorem~\ref{thm: monotonicity} is relatively straightforward, the technical details are complicated due to regularity issues.  The proof consists of two parts.  The first part is a calculation assuming sufficient regularity.  The second part is an approximation argument. For our later purposes, it is useful to have a very weak notion of (sub)-homogeneity.  We make the following technical definition

\begin{defn}\label{defn: weaklySubHomog}
    Fix constants $C, \delta >0$ and a non-decreasing function $\varpi : [0,\delta]\rightarrow \mathbb{R}_{>0}$ such that
    \[
    \int_{0}^{\delta} \frac{\varpi(s)}{s} ds <C
    \]
    Let  $((\Omega, g(x)dx), (\Omega', g'(y)dy), u, v)$ be an optimal transport map, where $g \in C^{1,\alpha}(\overline{\Omega})$ (resp. $g' \in C^{1,\alpha}(\overline{\Omega'}$). We say that $g(x)$ (resp. $g'(y)$) is {\bf weakly sub-homogeneous} of degree $l$ about $0$, with parameters $(C,\delta,\varpi)$ (resp. degree $k$) if, for all $s\leq \delta$ we have
    \[
    \int_{\{x\cdot \nabla u(x) \leq s\}}x\cdot \nabla g(x)dx \leq (l+\varpi(s))\int_{\{x\cdot \nabla u(x) \leq s\}}g(x)dx
    \]
    and
    \[
    \int_{\{y\cdot \nabla v(y) \leq s\}}y\cdot \nabla g'(y)dy \leq (k+\varpi(s))\int_{\{y\cdot \nabla v(y) \leq s\}}g'(y)dy.
    \]
    By convention, if $\varpi(s) \equiv 0$, then we shall say that $g$ (resp. $g'$) is weakly sub-homogeneous of degree $l$ about $0$ (resp. degree $k$). 
    \end{defn}

Before proceeding to the proof, let us give some examples of weakly sub-homogeneous functions.  The primary example we have in mind is when $g$ is non-negative and  satisfies 
\begin{equation}\label{eq: subhomogeneous}
x\cdot \nabla g(x) \leq lg(x).
\end{equation}
then clearly $g$ is weakly homogeneous of degree $l$ about $0$.  If $g$ is non-negative and homogeneous of degree $l$, then for any $\epsilon>0$, $g+\epsilon$ satisfies \eqref{eq: subhomogeneous}.   On the other hand, the notion of weak subhomogeneity is sufficiently general to allow positive, Dini continuous densities; see Lemma~\ref{lem: weak-sub-homogeneous-for-C-alpha-density} below.
\\

The following proposition is the key calculation, which can be viewed as a formal proof of Theorem~\ref{thm: monotonicity}.

\begin{prop}\label{prop: monotonicityRegularized}
    Let  $((\Omega, g(x)dx), (\Omega', g'(y)dy), u, v)$ be an optimal transport map with  where $(g(x), g'(y))$ are positive, bounded uniformly away from zero, $g\in  C^{2}(\overline{\Omega})$, $g' \in C^{2}(\overline{\Omega'})$. Assume that $\Omega, \Omega'$ are smooth and uniformly convex,  $0\in \overline\Omega\cap\overline{\Omega'}$. Assume in addition that $(g,g')$ are weakly sub-homogeneous of degree $(l,k)$ about $0$ with parameters $(C,\delta, \varpi)$ in the sense of Definition~\ref{defn: weaklySubHomog}. Define the quantity
	\[\chi(r) = r^{-\frac{2(n+l)}{1+\frac{n+l}{n+k}}}\mu(D_r(u,0, 0)) = r^{-\frac{2(n+k)}{1+\frac{n+k}{n+l}}}\nu(D_r(v,0, 0)),\]
     where we recall that $D_r(u,0,0)=\{x\cdot \nabla u(x)\leq r^2\}$ nd
     \[\mu(D_r(u, 0,0)) = \int_{\{x\cdot \nabla u(x)\leq r^2\}}g(x)dx.\] 
     Then there is a constant $\beta = \beta (n, k, l)$ such that $e^{-\beta \int_{0}^{r^2}\frac{\varpi(s)}{s}ds}\chi(r)$
    is monotone non-increasing in $r$ for $r^2 \leq \delta$.
\end{prop}

\begin{proof}
Since $\Omega, \Omega'$ are smooth, uniformly convex and $g,g'$ are positive, and $C^2$ regular up to the boundary, we have that $u\in C^{3,\alpha}(\Omega) \cap C^{2,\alpha}(\overline{\Omega})$ by the results of Urbas \cite{Urbas}. Let $f(x)=-\frac{1}{2}\log (g'(\nabla u(x))g(x))$ and define the Laplacian
\[
\begin{aligned}
\Delta_f \phi &:= \frac{1}{\sqrt{\det D^2u}\,e^{-f}}\del_i(u^{ij}\sqrt{\det D^2u}\, e^{-f}\del_j\phi)\\
&= \frac{1}{g}\del_i(u^{ij} g \del_j\phi)\\
&=u^{ij}\phi_{ij}-u^{ik}u^{pj}u_{ikp}\phi_j+u^{ij}\frac{1}{g}g_i\phi_j
\end{aligned}
\]
where $\del_i = \frac{\del}{\del x_i}$ and $\phi_i = \frac{\del \phi}{\del x_i}$.  Note that $\Delta_f$ satisfies the product formula
\[
\Delta_f(\varphi \psi) = \varphi \Delta_f\psi + \psi \Delta_f \varphi + 2u^{ij}\varphi_i\psi_j,
\]
as well as the integration by parts formula
\begin{equation}\label{eq: IBPweightLaplacian}
\begin{aligned}
\int_{\Omega} (\Delta_f\varphi) \psi \sqrt{\det D^2u}e^{-f} &= -\int_{\Omega} \langle \nabla \varphi, \nabla \psi\rangle_u\sqrt{\det D^2u}e^{-f}\\
\langle \nabla \varphi, \nabla \psi \rangle_u &:= u^{ij}\varphi_i\psi_j
\end{aligned}
\end{equation}
provided $\psi$ is compactly supported in $\Omega$.  More concretely, we can write the integral appearing in~\eqref{eq: IBPweightLaplacian} as
\[
\int_{\Omega} (\Delta_f\varphi) \psi \sqrt{\det D^2u}e^{-f}=\int_{\Omega} (\Delta_f\varphi) \psi\, g(x)dx
\]
from which the integration by parts formula follows trivially.

The main observation is that we have the following formulas:
\begin{equation}\label{eq: usefulFunctionsLinearizedMA}
\begin{aligned}
    \Delta_f u &=n+\sum_j\frac{1}{g'}\frac{\del g'}{\del y_j}u_j\\
    \Delta_f\left(x\cdot \nabla u(x) \right)&= 2n+\sum_j\frac{1}{g'}\frac{\del g'}{\del y_j}u_j+\sum_j\frac{1}{g}\frac{\del g}{\del x_j}x_j
\end{aligned}
\end{equation}
Let us prove the first formula.  From the expression for $\Delta_{f}$ we have
\[
\Delta_fu = n-u^{iq}u^{pj}u_{iqp}u_j + u^{ij}\frac{1}{g}g_iu_j
\]
By differentiating the optimal transport equation, we have
\[
u^{ik}u_{ikp} + \frac{1}{g'}\frac{\del g'}{\del y_m}u_{mp} =\frac{1}{g}\frac{\del g}{\del x_p}
\]
so that
\[
\begin{aligned}
u^{ik}u^{pj}u_{ikp}u_j&= u^{pj}u_j\frac{1}{g} g_p- \frac{1}{g'}\frac{\del g'}{\del y_j}u_j\\
\end{aligned}
\]
Therefore, we have
\[
\Delta_fu = n+ \sum_j\frac{1}{g'}\frac{\del g'}{\del y_j}u_j
\]
Next we consider the second formula in ~\eqref{eq: usefulFunctionsLinearizedMA}.  From the product formula, for each $1 \leq m \leq n$, we have
\[
\Delta_f (x_mu_m) = x_m\Delta_f u_m + u_m\Delta_f x_m + 2
\]
We compute
\[
\begin{aligned}
\Delta_fu_m &= \frac{1}{g}\frac{\del g}{\del x_m}\\
\Delta_f x_m &= -u^{pm}u^{iq}u_{ipq} + u^{im}\frac{1}{g}g_i = \frac{1}{g'}\frac{\del g'}{\del y_m},
\end{aligned}
\]
and the second formula follows.

Using these two functions we prove the monotonicity formula.  Let us denote by $\psi = x\cdot \nabla u(x)$ and consider the function $\phi= \gamma u + \psi $ for some $\gamma$ to be determined.  In the following calculation, for vectors $V,W$, we denote by $V\cdot W$ the euclidean inner product, and use $|V|$ for the euclidean norm of $V$.  We use 
\[
\langle V,W \rangle_u := V^iW^ju_{ij}, \qquad |V|_u^2 = \langle V, V \rangle_u.
\]
For $\kappa >0$ we compute
\[
\frac{d}{ds}\left(s^{-\kappa}\int_{\{\psi \leq s\}} e^{-f}\sqrt{\det D^2u}\right) = -\kappa s^{-(\kappa+1)}\int_{\{\psi \leq s\}}g(x)dx + s^{-\kappa}\int_{\{\psi =s\}}\frac{g(x)}{|\nabla \psi|}d\mathcal{H}^{n-1}(x)
\]
Now we use the formula
\[
\Delta_f \phi = \lambda + (1+\gamma)\left(\sum_j\frac{1}{g'}\frac{\del g'}{\del y_j}u_j-k\right) + \left(\sum_j\frac{1}{g}\frac{\del g}{\del x_j}x_j-l\right).
\]
where $\lambda =\gamma(n+k) +2n+k +l$. Then we have
\begin{equation}\label{eq: strictIneqMonotonicityFirstStep}
\begin{aligned}
\int_{\{\psi \leq s\}}g(x)dx &= \frac{1}{\lambda}\int_{\{\psi \leq s\}}(\Delta_f \phi)gdx - \frac{(1+\gamma)}{\lambda}\int_{\{\psi \leq s\}}\left(\sum_j\frac{1}{g'}\frac{\del g'}{\del y_j}u_j-k\right)\, g(x)dx\\
&\quad -\frac{1}{\lambda}\int_{\{\psi \leq s\}} \left(\sum_j\frac{1}{g}\frac{\del g}{\del x_j}x_j-l\right)g(x)dx
\end{aligned}
\end{equation}
We now use that $(g,g')$ are weakly sub-homogeneous of degree $(l,k)$ about $0$ with parameters $(C,\delta, \varpi)$.  This yields
\[
\int_{\{\psi \leq s\}} \left(\sum_j\frac{1}{g}\frac{\del g}{\del x_j}x_j-l\right)g(x)dx \leq \varpi(s)\int_{\{\psi \leq s\}}g(x)dx.
\]
Similarly, by the optimal transport equation we have
\[
\begin{aligned}
\int_{\{\psi \leq s\}}\left(\sum_j\frac{1}{g'}\frac{\del g'}{\del y_j}u_j-k\right)\, g(x)dx &= \int_{\{\psi \leq s\}}\left(\sum_j\frac{1}{g'}\frac{\del g'}{\del y_j}y_j-k\right)\, g'(y)dy\\
&\leq \varpi(s)\int_{\{\psi \leq s\}}g'(y)dy \\
&= \varpi(s)\int_{\{\psi \leq s\}}g(x)dx
\end{aligned}
\]
Thus, we have
\begin{equation}\label{eq: strictIneqMonotonicitySecondStep}
\begin{aligned}
\int_{\{\psi \leq s\}}g(x)dx &\geq \frac{1}{\lambda}\int_{\{\psi \leq s\}}(\Delta_f \phi)\,g(x)dx - \frac{(2+\gamma)}{\lambda}\varpi(s)\int_{\{\psi \leq s\}}g(x)dx\\
\end{aligned}
\end{equation}
It remains to analyze the first term. Let $\nu_{\del\Omega}$ denote the outward pointing normal vector to $\del \Omega$. We can now integrate-by-parts to get
\begin{equation}\label{eq: IBPsketchy}
\begin{aligned}
    \frac{1}{\lambda}\int_{\{\psi \leq s\}}(\Delta_f \phi)\,g(x)dx&= \frac{1}{\lambda}\int_{\Omega \cap \{\psi =s\}}\langle \nabla \phi, \nabla \psi \rangle_{u}\frac{g(x)}{|\nabla \psi|} \, d\mathcal{H}^{n-1}(x)\\
    &\quad +\frac{1}{\lambda} \int_{\del \Omega \cap \{\psi \leq s\}} \langle \nabla \phi, \nu_{\del \Omega}\rangle_u g(x)\, d\mathcal{H}^{n-1}(x)
\end{aligned}
\end{equation}
Now we compute
\[
\begin{aligned}
\langle \nabla \phi, \nabla \psi\rangle_u &= \gamma u^{ij} u_i(u_j+\sum_mx_mu_{mj}) +u^{ij}\left((u_i+\sum_px_pu_{pi})((u_j+\sum_mx_mu_{mj})\right)\\
&= (1+\gamma)|\nabla u|^2_u + |x|_u^2 +(2+\gamma)\psi
\end{aligned}
\]
where we recall that $|x|_u^2 = \sum_{i,m} u_{im}x_ix_m$.  Using that $\langle \nabla u, x\rangle_u= \psi$, we can complete the square as follows
\[
\langle \nabla \phi, \nabla \psi\rangle_u  = |(\sqrt{1+\gamma}) \nabla u -x|^2_{u} +(1+\sqrt{1+\gamma})^2\psi.
\]
In total, we have
\[
\begin{aligned}
    \frac{d}{ds}\left(s^{-\kappa}\int_{\Omega \cap \{\psi \leq s\}} g(x)dx\right) &\leq -\frac{\kappa}{\lambda}s^{-(\kappa+1)}\int_{\Omega \cap \{\psi =s\}}|(\sqrt{1+\gamma}) \nabla u -x|^2_{u} \frac{g(x)}{|\nabla \psi|}\, d\mathcal{H}^{n-1}(x)\\
    &-\frac{\kappa}{\lambda}s^{-(\kappa+1)}\int_{\Omega\cap\{\psi =s\}}\left((1+\sqrt{1+\gamma})^2-\frac{\lambda}{\kappa}\right)\psi\frac{g(x)}{|\nabla \psi|}\, d\mathcal{H}^{n-1}(x)\\
    &\quad -\frac{\kappa}{\lambda}s^{-(\kappa+1)}\int_{\del \Omega \cap \{\psi \leq s\}} \langle \nabla \phi, \nu_{\del \Omega}\rangle_u g(x)\, d\mathcal{H}^{n-1}(x)\\
    &+\frac{\kappa(2+\gamma)}{\lambda} s^{-(\kappa+1)} \varpi(s)\int_{\{\psi \leq s\} } g(x) dx
\end{aligned}
\]
Choose $\kappa= \kappa(\gamma)$ so that $\kappa\lambda^{-1} = (1+\sqrt{1+\gamma})^{-2}$; that is
\[
\kappa(\gamma) = \frac{\gamma(n+k) +2n+k +l}{(1+\sqrt{1+\gamma})^2}.
\]
and now choose $\gamma = \gamma_*>-1$ so that $\kappa_*= \kappa(\gamma_*)$ is minimized.  We have
\begin{equation}\label{eq: kappagammastar}
\kappa_{*} = \frac{(n+k)(n+l)}{(n+k)+(n+l)} \qquad \sqrt{1+\gamma_*} = \frac{n+l}{n+k}
\end{equation}
For these choices we obtain
\begin{equation}\label{eq: monotonicityVersion1}
\begin{aligned}
    \frac{d}{ds}\left(s^{-\kappa_*}\int_{\Omega\cap \{\psi \leq s\}} gdx\right) &\leq  -(1+\sqrt{1+\gamma_*})^{-2}s^{-(\kappa_*+1)}\left(\int_{\Omega \cap \{\psi =s\}}|(\sqrt{1+\gamma_*}) \nabla u -x|^2_{u}\frac{g}{|\nabla \psi|}\right)\\
    &\quad -(1+\sqrt{1+\gamma_*})^{-2}s^{-(\kappa_*+1)}\int_{\del \Omega \cap \{\psi \leq s\}} \langle \nabla \phi, \nu_{\del \Omega}\rangle_u g\\
    &\quad +(1+\sqrt{1+\gamma_*})^{-2}(2+\gamma_*)\frac{\varpi(s)}{s}\left(s^{-\kappa_*}\int_{\Omega\cap \{\psi \leq s\}} gdx\right)
\end{aligned}
\end{equation}
We now analyze the boundary term.  We have
\[
\langle \nabla \phi, \nu_{\del \Omega}\rangle_u= (1+\gamma_*)u_iu^{ij}(\nu_{\del\Omega})_j + x_j(\nu_{\del\Omega})_{j}
\]
Since $\Omega$ is convex, and $0\in \overline{\Omega}$ we have
\[
x\cdot \nu_{\del\Omega} \geq 0
\]
with equality if and only if $0\in \del \Omega$ and $\Omega$ is conical near $0$.  Similarly, we note that $u^{ij}(\nu_{\del\Omega})_j(x)$ is normal to $\del \Omega'$ at $y=\nabla u(x)$.  Thus, since $0 \in \overline{\Omega'}$,  the convexity of $\Omega'$ yields
\[
u_iu^{ij}(\nu_{\del\Omega})_j(x) \geq 0
\]
with equality if and only if $0\in \del \Omega'$ and $\Omega'$ is conical near $0$.  Thus, since $\int_0^{\delta} \frac{\varpi(s)}{s}ds < +\infty$, we obtain the desired volume monotonicity by substituting $r^2 = s$. 
\end{proof}

\begin{remark}\label{rk: FormalProof}
      When $\Omega, \Omega'$ are $C^{2}$ and uniformly convex, $g>0$ on $\overline{\Omega}$, and $g'>0$ on $\overline{\Omega'}$, and $(g,g')$ are homogeneous of degree $(l,k)$ the proof of Proposition~\ref{prop: monotonicityRegularized} yields the desired monotonicity formula, but the equality case can never occur due to the strict convexity of the boundaries.
\end{remark}

\begin{remark}\label{rk: regularityIBP}
In the proof of Proposition~\ref{prop: monotonicityRegularized}, the $C^{2,\alpha}$ boundary regularity of $u,v$ is only needed to justify the integration by parts formula~\eqref{eq: IBPsketchy}.  In particular, we observe that $C^{2,\alpha}$ boundary regularity is only needed near points of $\del (\Omega \cap \{\psi \leq r^2\})$.
\end{remark}

We now prove Theorem~\ref{thm: monotonicity}.  In order to deduce this result from Proposition~\ref{prop: monotonicityRegularized} we need an approximation result.  While such results are standard in the literature, we state a self-contained result which will be applied several times throughout the paper.

\begin{defn}\label{defn: localSmoothApprox}
    Given an optimal transport map $((\Omega, g(x)dx), (\Omega', g'(y)dy), u, v)$ with $0\in \overline{\Omega}$ and $0 = u(0) = |\nabla u|(0)$. We say that a sequence $((\Omega_{\eps}, g_{\eps}(x)dx), (\Omega'_{\eps}, g'_{\eps}(y)dy), u_{\eps}, v_{\eps})$ with $0\in \overline{\Omega}_{\eps}$  is a {\bf local smooth approximation} of $((\Omega, g(x)dx), (\Omega', g'(y)dy), u, v)$ near $0$ at some scale $r>0$ if
    \begin{enumerate}
        \item $D_r(u,0)\subset \Omega_{\eps}$ and $D_r(v,0)\subset \Omega_{\epsilon}'$ where $\Omega_{\eps}, \Omega_{\eps}'$ are compact, smoothly bounded, uniformly convex. 
        \item $g_{\eps}(x), g'_{\eps}(y)$ are positive and smooth in $\overline\Omega_{\eps}, \overline{\Omega'}_{\eps}$ respectively and 
        \[g_{\eps}(x)\chi_{D_r(u, 0)}(x)dx\to g(x)\chi_{D_r(u, 0)}(x)dx\]  
        \[g'_{\eps}(y)\chi_{D_r(v, 0)}(y)dy\to g'(y)\chi_{D_r(v, 0)}(y)dy\]
        in the weak sense of measures.  
        \item $(u_{\eps}, v_{\eps})$ are smooth in $(\overline \Omega_{\eps}, \overline{\Omega'}_{\eps})$ respectively, and for any $r'<r$, we have $D_{r'}(u_{\eps}, 0)\to D_{r'}(u, 0)$ and $D_{r'}(v_{\eps}, 0)\to D_{r'}(v, 0)$, and for some $\delta>0$, we have the convergence $ u_{\eps}\to  u$ in $C^{1, \delta}(\overline{D_{r'}(u, 0)})$ and $ v_{\eps}\to v$ in $C^{1, \delta}(\overline{D_{r'}(v, 0)})$ . 
    \end{enumerate}
\end{defn}

    \begin{lem}\label{lem: existApprox}
    Given an optimal transport map $((\Omega, g(x)dx), (\Omega', g'(y)dy), u, v)$ with $0 \in \overline\Omega$ and $0 = u(0) = |\nabla u|(0)$. 
    \begin{enumerate}
        \item Suppose that $C^{-1}\leq g(x), g'(y)<C$
        and both $g(x), g'(y)$ have a modulus of continuity $\theta(d)$. Then for any $r>0$, the optimal transport map admits a local smooth approximation $((\Omega_{\eps}, g_{\eps}(x)dx), (\Omega'_{\eps}, g'_{\eps}(y)dy), u_{\eps},  v_{\eps})$ of scale $r$ such that 
        \[
        \begin{aligned}
            C^{-1}-o(1)\leq &g_{\eps}(x)\leq C+o(1)\quad \text{ in }\Omega_{\eps}\cap D_r(u_{\eps}, 0)\\
            C^{-1}-o(1)\leq &g'_{\eps}(y)\leq C+o(1) \quad \text{ in }\Omega'_{\eps}\cap D_r(v_{\eps}, 0)
            \end{aligned}
            \]
        and $g_{\eps}$ (resp. $g_{\eps}'$) has a modulus of continuity of $(1+o(1))\theta(d)$ in $ D_r(u, 0)$ (resp. $D_r(v, 0)$)  where $o(1)$ is a term that goes to $0$ as $\eps\to 0$. 
        \item  Suppose that $(g, g')\in (C^{\alpha}(\overline{\Omega}), C^{\alpha}(\overline{\Omega'}))$, are positive on $(\Omega, \Omega')$, and homogeneous of degree $(l, k)$. Then for any scale $r>0$, the optimal transport map admits a local smooth approximation $((\Omega_{\eps}, g_{\eps}(x)dx), (\Omega'_{\eps}, g'_{\eps}(y)dy), u_{\eps}, v_{\eps})$ at scale $r$ such that $(g_{\eps}, g'_{\eps})$ are pointwise weakly sub-homogeneous of degree $(l, k)$ about $0$ in $D_r(u, 0)$. That is
        \[
        \begin{aligned}
        x\cdot \nabla g_{\eps}(x)&\leq lg_{\eps}(x) \quad \text{ in } D_r(u, 0)\\
        y\cdot \nabla g_{\eps}'(y)&\leq kg_{\eps}'(y) \quad \text{ in } D_r(v, 0).
        \end{aligned}
        \]
    \end{enumerate}
\end{lem}

\begin{proof}
    We prove the lemma in several steps.  Define
     \[
    A_{r}:={\rm Convex Hull}(D_r(u,0)) \qquad A_r' = {\rm Convex Hull}(D_r(v,0))
    \]
    Let $\Omega_{\epsilon}$ (resp. $\Omega_{\epsilon}'$) be a sequence of smoothly bounded, uniformly convex domains approximating $A_r$ (resp. $A_r'$) from the outside and such that for any $x\in \partial\Omega\cap \overline{D_r(u, 0)}$, we have $d(x, \partial\Omega_{\eps})\leq \eps$. 
    
    Now let $\bar{u}$ (resp. $\bar{v}$) be the minimal convex extensions of $u\big|_{D_{r}(u,0)}$ (resp. $v\big|_{D_{r}(v,0)}$). Since $u \in C^{1,\delta}(\overline{D_r(u,0)})$ (resp. $v \in C^{1,\delta}(\overline{D_r(v,0)})$ and $\nabla u: D_r(u,0) \rightarrow D_{r}(v,0)$
    we can view $(\overline{u},\overline{v})$ as a solution of the optimal transport problem from $(\Omega_{\epsilon}, \chi_{D_{r}(u,0)}g(x)dx)$ to $(\Omega_{\epsilon}', \chi_{D_{r}(v,0)}g'(y)dy)$ for any $\epsilon$.  We now treat the two cases separately.
\bigskip

{\bf Case (1)}:  Suppose $C^{-1}\leq g(x), g'(y)<C$
        and both $g(x), g'(y)$ have a modulus of continuity $\theta(d)$.  Extend $g,g'$ to functions $\bar g, \bar g'$ which are defined on all of $\mathbb{R}^n$, while preserving the modulus of continuity.  Next we mollify to get $\overline g_{\eps} = \eta_{\eps}\star\overline g$, $\overline g_{\eps}' = \eta_{\eps}\star\overline g'$, where $\eta_{\eps}$ is the standard mollifier. The functions $\bar g_{\eps}, \bar g_{\eps}'$ are smooth and since mollification preserves the modulus of continuity,  $g$ and $g'$ still have modulus of continuity $\theta(d)$. Now we can choose $g_{\eps}:\overline{\Omega_{\eps}}\to \mathbb R$ to be a positive, bounded, smooth function such that
        \[g_{\eps}(x) = \begin{cases}
        \overline g_{\eps}(x)&\text{ for }x\in \overline{B_{2\eps}(D_r(u, 0)})\\
        \eps &\text{ for }x\in \overline{\Omega_{\eps}}\setminus B_{4\eps}(D_{r}(u, 0)). 
        \end{cases}\]
        We can easily see that $g_{\eps}(x)\rightarrow  g(x)\chi_{D_r(u, 0)}(x)$ uniformly as $\eps\to 0$. Moreover $g_{\epsilon}(x)$ still satisfies 
        \[C^{-1}\leq g_{\eps}(x)\leq C
        \text{ in }\overline{D_r(u, 0)}\]
        and has modulus of continuity $\theta(d)$ in $D_r(u, 0)$. We define $g'_{\eps}(y)$ in the same way, but with an extra constant factor $c_{\eps} = 1+o(1)$, which is chosen so that $\int_{\Omega_{\eps}}g_{\eps}(x)dx = \int_{\Omega_{\eps}'}g'_{\eps}(y)dy$. It follows that $g'_{\eps}$ satisfy the same upper/lower bounds and has the same modulus of continuity up to an additional $o(1)$ term.
\bigskip

{\bf Case (2)}: If $(g, g')\in(C^{\alpha}(\overline{\Omega}),C^{\alpha}(\overline{\Omega'}))$ are homogeneous of degree $(l, k)$, then we first extend them to global $C^{\alpha}$,  homogeneous functions $\overline g, \overline{g}'$ which are non-negative and defined on all of $\mathbb R^n$. The functions $\overline g+\eps$ and $\overline{g'}+\eps$ are sub-homogeneous, and in fact
        \[x\cdot \nabla(\overline g+\eps)- l(\overline g+\eps)\leq -l\eps<0. \]
        Next we mollify to get $\overline g_{\eps} = \eta_{\delta}\star(\overline g+\eps)$, $\overline g_{\eps}' = \eta_{\delta}\star(\overline g'+\eps)$, then for $\delta$ sufficiently small, $\overline g_{\eps}, \overline g'_{\eps}$ are still subhomogeneous of degree $(l, k)$, and we can define $g_{\eps}$ and $g'_{\eps}$ just as in case (1). 

\bigskip
Let $(u_{\eps}, v_{\eps})$ be the solution of the optimal transport problem from $(\Omega_{\eps}, g_{\eps}(x)dx)$ to $(\Omega'_{\eps}, g'_{\eps}(y)dy)$.  Then by the results of Urbas \cite{Urbas}, $u_{\eps}$ and $v_{\eps}$ are both smooth in $\Omega_{\eps}$ and $\Omega'_{\eps}$ respectively. Furthermore, from the stability of optimal transport maps we know that $u_{\eps}\to u$ uniformly in $D_r(u, 0)$ as $\eps\to 0$, and similarly for $v_{\eps}$. Finally, by Caffarelli's $C^{1, \delta}$-theory and its improvement by Jhaveri-Savin (see Proposition~\ref{prop: C1-delta-estimate}), we can upgrade the convergence of $u_{\eps}\to u$ to $C^{1, \delta}(\overline{D_{r'}(u, 0)})$ for any $r'<r$ and similarly for $v_{\epsilon} \rightarrow v$.
\end{proof}

\begin{proof}[Proof of Theorem~\ref{thm: monotonicity}]
   By part (2) of Lemma~\ref{lem: existApprox}, for any $r>0$ we can find a local smooth approximation 
   \[
   ((\Omega_{\eps}, g_{\eps}(x)dx), (\Omega'_{\eps}, g'_{\eps}(y)dy), u_{\eps}, v_{\eps})
   \]
   in the sense of Definition~\ref{defn: localSmoothApprox}.  We can now apply Proposition~\ref{prop: monotonicityRegularized} with $\varpi \equiv 0$ to deduce that, for any $r_2<r_1 <r$ we have
   \begin{equation}\label{eq: coercivityFormula}
\begin{aligned}
 &r_1^{-2\kappa_{*}}\int_{\Omega_{\epsilon} \cap \{\psi_{\epsilon}<r_1^2\}} g_{\epsilon}(x)dx - r_2^{-2\kappa_{*}}\int_{\Omega_{\epsilon} \cap \{\psi_{\epsilon}<r_2^2\}} g_{\epsilon}(x)dx\\
 &\leq -(1+\sqrt{1+\gamma_*})^{-2}\int_{r_2^2}^{r_1^2}s^{-(\kappa_*+1)}\int_{\Omega_{\epsilon} \cap \{\psi_{\epsilon}=s\}}|(\sqrt{1+\gamma_*})\nabla u_{\epsilon} -x|^2_{u_{\epsilon}}\frac{g_{\epsilon}(x)}{|\nabla \psi_{\epsilon}|}\, ds\, d\mathcal{H}^{n-1}(x)\\
 &=-(1+\sqrt{1+\gamma_*})^{-2}\int_{D_{r_1}(u_{\epsilon},0,0)\setminus D_{r_2}(u_{\epsilon},0,0)}\frac{|(\sqrt{1+\gamma_*})\nabla u_{\epsilon} -x|^2_{u_{\epsilon}}}{\psi_{\epsilon}^{\kappa_*+1}}g_{\epsilon}(x)dx
 \end{aligned}
\end{equation}
where $\kappa_*,\gamma_*$ are given by~\eqref{eq: kappagammastar}, and $\psi_{\epsilon} = x\cdot \nabla u_{\epsilon}(x)$.   We need to justify that we can pass to the limit.  The main complication is that
\[
|(\sqrt{1+\gamma_*})\nabla u_{\epsilon} -x|^2_{u_{\epsilon}}= (1+\gamma_{*})u_{\epsilon}^{ij}(u_{\epsilon})_i(u_{\epsilon})_j -2\sqrt{1+\gamma_*} (u_{\epsilon})_ix_i + (u_{\epsilon})_{ij}x^ix^j
\]
is a nonlinear expression involving the second derivatives of $u_{\epsilon}$, which is only controlled in $C^{1,\delta}$.  Nevertheless, this can be dealt with via standard techniques.  The most straightforward approach is to appeal to Caffarelli's interior $W^{2,p}$ estimates \cite{Caffarelli4} which yields  enough regularity to pass to the limit.  We  give a different, self-contained approach here exploiting convexity in a crucial way.  The main observation is that, if $A_{ij}$ is a positive definite symmetric matrix, and $V_i$ is a vector, then
\[
V_iA^{ij}V_j = \sup_{W} 2\langle W,V\rangle - \langle AW,W \rangle
\]
In particular, we can write
\[
\int_{U_{\epsilon}}\frac{|(\sqrt{1+\gamma_*})\nabla u_{\epsilon} -x|^2_{u_{\epsilon}}}{\psi_{\epsilon}^{\kappa_*+1}}g_{\epsilon}(x)dx= \sup_{W\in C^{\infty}_{c}(U_{\epsilon})}\int_{U_{\eps}} \left(2(V_{\epsilon})_i W_i - W_i(u_{\epsilon})_{ij}W_j\right) \frac{g_{\epsilon}(x)}{\psi_{\epsilon}^{\kappa_*+1}}dx
\]
where $U_{\epsilon} = D_{r_1}(u_{\epsilon},0,0)\setminus D_{r_2}(u_{\epsilon},0,0)$, and $(V_{\epsilon})_i = (\sqrt{1+\gamma_*})(u_{\epsilon})_i -(u_{\epsilon})_{ij}x_j$.

Let $W$ be a smooth vector field, with $K={\rm supp}(W) \Subset D_{r_1}(u,0,0)\setminus D_{r_2}(u,0,0)$. Since $u$ is strictly convex and $\nabla u(0)=0$, and by our choice of approximation we know $U_{\epsilon}= D_{r_1}(u_{\epsilon},0,0)\setminus D_{r_2}(u_{\epsilon},0,0)$ converge in the Hausdorff sense to $U:= D_{r_1}(u,0,0)\setminus D_{r_2}(u,0,0)$.  So, for  $\epsilon$ sufficiently small we have $K \subset D_{r_1}(u_{\epsilon},0,0)\setminus D_{r_2}(u_{\epsilon},0,0)$.  Since the $u_{\epsilon}$ converge to $u$ in $C^{1,\delta}$, the measures $d\mu_{\epsilon} =D^2u_\epsilon(x)dx$ converge to the finite Radon measure $d\mu = D^2u$ in the weak-$*$ topology and so $(u_{ij})_{\eps}$ converges weakly to $u_{ij}$ and $V_{\epsilon}$ converges weakly to $V$.  Thus, we have
\begin{equation}\label{eq: supLinFuncChar}
\begin{aligned}
&\int_{K} \left(2V_i W_i - W_iu_{ij}W_j\right)  \frac{g}{\psi^{\kappa_*+1}} dx \\
&= \liminf_{\epsilon \rightarrow 0}\int_{K} \left(2W_i(V_{\epsilon})_i - W_i(u_{\epsilon})_{ij}W_j\right) \frac{g_{\epsilon}}{\psi_{\epsilon}^{\kappa_*+1}} dx  \\
&\leq \liminf_{\epsilon \rightarrow 0}\int_{D_{r_1}(u_{\epsilon},0,0)\setminus D_{r_2}(u_{\epsilon},0,0)} \frac{|(\sqrt{1+\gamma_*})\nabla u_{\epsilon} -x|^2_{u_{\epsilon}}}{\psi_{\epsilon}^{\kappa_*+1}}g_{\epsilon}(x)dx
\end{aligned}
\end{equation}
Thus, we obtain
\[
\begin{aligned}
&r_1^{-2\kappa_{*}}\int_{\Omega \cap \{\psi<r_1^2\}} g(x)dx - r_2^{-2\kappa_{*}}\int_{\Omega \cap \{\psi<r_2^2\}} g(x)dx\\
&\leq -(1+\sqrt{1+\gamma_*})^{-2}\sup_{W\in C^{\infty}_{c}(U)}\int_{U} \left(2V_i W_i - W_iu_{ij}W_j\right)  \frac{g}{\psi^{\kappa_*+1}} dx
\end{aligned}
\]
and since the right hand side is non-positive (since $W\equiv 0$ is a competitor for the supremum), we obtain the monotonicity.  It only remains to address the equality case.  Suppose that $\chi(r)$ is constant on $[0,r_0]$.  Then, for any $W\in C^{\infty}_{c}(U)$ we have
\[
\int_{U} \left(2V_i W_i - W_iu_{ij}W_j\right)  \frac{g}{\psi^{\kappa_*+1}} dx\leq 0.
\]
Testing this inequality against $\pm \epsilon W$ for $\epsilon \rightarrow 0$ we deduce that 
\[
\frac{g}{\psi^{\kappa_*+1}}((\sqrt{1+\gamma_*})u_i -u_{ij}x_j) \equiv 0
\]
in the sense of measures.  As $\frac{g}{\psi^{\kappa_*+1}}>0$ in the interior, this implies $(\sqrt{1+\gamma_*})u_i -u_{ij}x_j \equiv 0$ in the sense of measures and distributions.  We can rewrite this as
\[
\del_i \left((1+\sqrt{1+\gamma_*})u-x_ju_j\right) \equiv 0
\]
in the sense of distributions.  The homogeneity of $u$ follows immediately.
\end{proof}

We can also prove an effective version of the monotonicity formula for Dini continuous densities. 
\begin{prop}\label{prop: effective-monotonicity-for Holder-density}
    Given an optimal transport map $((\Omega, g(x)dx), (\Omega', g'(y)dy), u, v)$ where the densities $(g(x), g'(y))$ satisfy
	\[C^{-1}\leq g(x), g'(y)\leq C\]
	and have modulus of continuity $\theta(d)$ satisfying 
    \[
    \int_0^1 \frac{\theta(s)}{s} ds < +\infty
    \]
   Then there exist $r_0>0, \delta>0$ depending only on $n$, and $C$ such that for any $x_0\in\overline\Omega$, the quantity
    \[\chi(x_0, r) = {\rm exp}\left(-\frac{1}{\delta}\int_{0}^{r^{\delta}}\frac{\theta(s)}{s}ds\right)\cdot r^{-n}\int_{D_r(u, x_0)}g(x)dx\]
    is monotone non-increasing in $r$ for $r\in (0, r_0]$. 
\end{prop}
The key is the following lemma, which establishes the weak sub-homogeneity of non-degenerate Dini continuous densities at small (but fixed) scales. 
\begin{lem}\label{lem: weak-sub-homogeneous-for-C-alpha-density}
Fix an optimal transport map $((\Omega, g(x)dx), (\Omega', g'(y)dy), u, v)$ where the densities $(g(x), g'(y))$ are smooth and satisfy
	\[C^{-1}\leq g(x), g'(y)\leq C\]
	and have modulus of continuity $\theta(d)$. 
    Then there exist constants $\delta, r_0>0$, depending only on $n, C$ such that we have
    \[\int_{D_r}(x\cdot \nabla g)dx \leq 2nC\theta(r^\delta)\int_{D_r}g(x)dx.\]
    for all $r<r_0$.  That is $(g,g')$ are subhomogeneous of degree $(0,0)$ with parameters $(C',r_0,\theta(r^\delta)) $
\end{lem}
\begin{proof}
    We write the integral in radial coordinates
    \begin{equation}\label{eqn: integral-in-radial-coord}
        \int_{D_r}(x\cdot \nabla g)dx = c_n\int_{t\in S^{n-1}}\int_0^{\phi(t)}(s(\nabla_t g)(st))s^{n-1}\,dsd\sigma(t)
    \end{equation}
    where $\phi(t)$ is chosen so that $(\phi(t)t)\cdot \nabla u(\phi(t)t) = r^2$. We can integrate by parts along radial rays to get rid of the derivative on $g$
    \[\int_0^\phi g'(s)s^nds = -n\int_0^\phi g(s)s^{n-1}ds+g(\phi)\phi^{n} = n\int_0^\phi(g(\phi)-g(s))s^{n-1}ds\leq n ({\rm osc}_{D_r} g)\int_0^\phi s^{n-1}ds.\]
    Putting this into equation~\eqref{eqn: integral-in-radial-coord}, we get 
    \[\int_{D_r}(x\cdot \nabla g)dx \leq n\theta(\text{diam}(D_r))\int_{D_r}dx.\]
    The constant $C$ controls the doubling constant of $(g, g')$, hence by the uniform convexity estimate (Proposition~\ref{prop: C1-delta-estimate}), we have \[D_r(u, 0)\subset S_{r^2}(u, 0)\subset B_{r^{\delta}}(0)\] for some $\delta>0$ and all $r <r_0$ depending on $C$. Thus we have
    \[
    \int_{D_r}(x\cdot \nabla g)dx \leq 2nC\theta(r^{\delta})\int_{D_r}g(x)dx.
    \]
\end{proof}

\begin{proof}[Proof of Proposition~\ref{prop: effective-monotonicity-for Holder-density}]
    If $(g, g')$ are smooth and $(\Omega, \Omega')$ have smooth boundaries, then this follows from combining Lemma~\ref{lem: weak-sub-homogeneous-for-C-alpha-density} and Proposition~\ref{prop: monotonicityRegularized}. For the general case, we apply Lemma~\ref{lem: existApprox} part (1) to find a smooth approximation at scale $r<r_0$ preserving the modulus of continuity (up to $o(1)$ error). We can then apply the same approximation argument as in the proof of Theorem~\ref{thm: monotonicity} for the case $l=k=0$. It only remains to observe that
    \[
    \int_0^{r_0} \frac{\theta(s^{\delta})}{s}ds = \frac{1}{\delta}\int_0^{r_0^{\delta}} \frac{\theta(s)}{s}ds.
    \]
\end{proof}

\section{Homogeneity of Blow-ups}\label{Sec: homogeneity-of-blow-ups}
    In this section, we use the monotonicity formula to show that the blow-ups of optimal transport maps are always homogeneous. We will treat two cases, the first when the densities are non-degenerate and Dini continuous, and the second case when the densities are homogeneous. 
    
    First let us define what a blow-up is. 
    \begin{defn}
    $((\Omega, g(x)dx), (\Omega', g'(y)dy), u, v)$ be an optimal transport map with $0\in\overline{\Omega}\cap\overline{\Omega'}$ such that $u(0) = |\nabla u|(0) = 0$. We call an optimal transport map \[((\Omega_{\infty}, g_{\infty}(x)dx), (\Omega_{\infty}', g'_{\infty}(y)dy), u_{\infty}, v_{\infty})\] a {\bf blow-up} of $((\Omega, g(x)dx), (\Omega', g'(y)dy), u, v)$ at $0\in \overline\Omega$ along scales $h_i\to 0$, if there is $R>1$ and a sequence of positive symmetric matrices $A_{h_i}$ satisfying
    \[B_{R^{-1}}(0)\subset A_{h_i}(S_{h_i}^c(u, 0))\subset B_R(0)\]
    and a sequence of positive constants $c_i>0$, such that the sequence of rescaled optimal transport maps $((\Omega_{i}, g_i(x)dx), (\Omega'_{i}, g'_i(y)dy), u_i, v_i)$ given by
	\[\Omega_i := A_{h_i}(\Omega), \qquad \Omega_i' := h_i^{-1}A_{h_i}^{-1}(\Omega'),\]
	\[g_i(x) := c_ig(A_{h_i}^{-1}x), \qquad g'_i(y) := c_ih_i^n(\det A_{h_i})^2g'(h_iA_{h_i}y),\]
	\[u_i(x) := \frac{u(A_{h_i}^{-1}x)}{h_i}, \qquad v_i(y) := \frac{v(h_i A_{h_i}y)}{h_i},\]  converges to $((\Omega_{\infty}, g_{\infty}(x)dx), (\Omega_{\infty}', g'_{\infty}(y)dy), u_{\infty}, v_{\infty})$ in the following sense
    \begin{enumerate}
        \item $(\Omega_i, \Omega'_i)\to (\Omega_{\infty}, \Omega_{\infty}')$ in the locally Hausdorff sense. 
        \item $(g_i(x)dx, g_i'(y)dy)\to (g_{\infty}(x)dx, g_{\infty}'(y)dy)$ in the weak sense of measures. 
        \item The minimal convex extensions $(\bar u_i, \bar v_i)$ converges to the minimal convex extensions $(\bar u_{\infty}, \bar v_{\infty})$ in $C^{1, \alpha}_{loc}(\mathbb R^n)$. 
    \end{enumerate}
\end{defn}
    \subsection{With Dini continuous density}
    In this section, we prove the existence and homogeneity of blow-ups when $(g(x), g'(y))$ are non-degenerate and Dini continuous. We first recall some basic properties of Dini continuous functions.  

    \begin{defn}
        Suppose that $g$ is a continuous function on $\overline{\Omega}$, with modulus of continuity $\theta(d)$. We define
        \[
        \|g\|_{{\rm Dini},\Omega} := \|g\|_{L^{\infty}(\overline{\Omega})} + [g]_{{\rm Dini},\overline{\Omega}}
        \]
        where
        \[
        [g]_{{\rm Dini},\overline{\Omega}} := \int_0^{\min\{1, {\rm diam}(\Omega)\}} \frac{\theta(s)}{s}ds
        \]
    \end{defn}

    The Dini norm has the following compactness property, whose proof is elementary.
    \begin{lem}\label{lem: Dini-compactness}
    If $\{g_k\}_{k\in \mathbb{N}}$ is a sequence of continuous functions on $\overline{\Omega}$ with $\|g_k\|_{{\rm Dini},\overline{\Omega}} \leq M$, then there is a subsequence convergining uniformly to a continuous function $g_{\infty}$.  Furthermore, $g_{\infty}$ satisfies
    \[
    [g_\infty]_{{\rm Dini},\overline{\Omega}} \leq \liminf_{k'\rightarrow \infty}[g_{k'}]_{{\rm Dini}, \overline{\Omega}}
    \]
    \end{lem}

    We first prove a uniform density estimate when $(g(x), g'(y))$ are non-degenerate and Dini continuous on general convex domains. This estimate improves upon the previously known uniform density estimates (\cite[Section 3]{Caffarelli3}, \cite[Section 2]{Chen-Liu-Wang}), which required additional regularity assumptions on the boundary of $\partial\Omega$.  

    \begin{lem}[Uniform density]\label{lem: uniform density}
    Given an optimal transport map $((\Omega, g(x)dx), (\Omega', g'(y)dy), u, v)$ where the densities $(g(x), g'(y))$ satisfy
	\[C^{-1}\leq g(x), g'(y)\leq C\]
	and
	\[[g]_{{\rm Dini}, \overline{\Omega}}+[g']_{{\rm Dini}, \overline{\Omega'}}\leq C\]
    for some $C>1$. Then for any $x_0\in \overline\Omega$ and $0<h\leq 1$, we have
    \[\frac{|S_h(u, x_0)|}{h^{\frac n 2}}\geq c|S_1(u, 0)|\]
    for $c>0$ depending only on $n, C$. 
\end{lem}
    \begin{proof}
        By Proposition~\ref{prop: effective-monotonicity-for Holder-density}, there exist $A$, $h_0$ depending on $C$ such that for any $0<h<h_0$, we have 
        \[h^{-\frac n 2}\mu(D_{h^{\frac 1 2}}(u, 0))\geq e^{-A}h_0^{-\frac n 2}\mu(D_{h_0^{\frac 1 2}}(u, 0)).\]
        Hence
        \[|D_{h^{\frac 1 2}}(u, 0)|\geq c\mu(D_{h^{\frac 1 2}}(u, 0))\geq ce^{-A}h_0^{-\frac n 2}\mu(D_{h_0^{\frac 1 2}}(u, 0))h^{\frac n 2}\geq ch^{\frac n 2}|D_{h_0^{\frac 1 2}}(u, 0)|.\]
        Together with Proposition~\ref{prop: comparison between sections and extrinsic ball}, we have
        \[|S_h(u, 0)|\geq ch^{\frac n 2}|S_{h_0}(u, 0)|,\]
        and by convexity, we also have
        \[|S_{h_0}(u, 0)|\geq h_0^{n}|S_1(u, 0)|\]
        which gives the desired estimate. 
    \end{proof}

    \begin{theorem}[Homogeneity of blow-ups]\label{thm: homog-of-blowup-lebesgue-measure}
	   Let $((\Omega, g(x)dx), (\Omega', g'(y)dy), u, v)$ be an optimal transport map with $0\in \overline\Omega \cap \overline{\Omega'}$, and suppose $g(x), g'(y)$ are positive and Dini continuous on $\overline\Omega, \overline{\Omega'}$ respectively. Then for any sequence of scales $h_i\to 0$, there is a subsequence along which a blow-up exists and is of the form $((\mathtt C, cdx), (\mathtt C', c'dy), u_{\infty}, v_{\infty})$ for some constants $c, c'>0$. Moreover, $(\mathtt C, \mathtt C')$ are cones about the origin and $(u_{\infty}, v_{\infty})$ are homogeneous of degree $2$. 
	\end{theorem}
        \begin{proof}By Lemma~\ref{lem: uniform density}, we know that,  
        \[|S_h(u, 0)|\geq ch^{\frac n 2},\qquad |S_h(v, 0)|\geq ch^{\frac n 2}.\] Since the $(g(x), g'(y))$ are both doubling measures, Proposition~\ref{prop: properties-of-sections} also implies \[|S_h(u, 0)||S_h(v, 0)|\leq Ch^n,\] and hence we have
        \[|S_h^c(u, 0)|\sim |S_h^c(v, 0)|\sim |S_h(u, 0)|\sim|S_h(v, 0)|\sim h^{\frac n 2}.\]
        Hence without loss of generality, we can assume the normalizing matrices are given by $A_h = h^{-\frac 1 2}M_h$ with $\det M_h = 1$. It follows that for any $R>1$, the rescaled domains $\Omega_i\cap B_R(0) = A_{h_i}(\Omega)\cap B_R(0)$ have a uniform lower bound on its volume, and therefore we can extract a subsequence for which the convex domains $(\Omega_i, \Omega'_i)$ converges to convex domains $(\mathtt C, \mathtt C')$ locally in the Hausdorff sense, and moreover $(\mathtt C, \mathtt C')$ have non-empty interior. 
        The rescaled potentials
        \[u_i(x) := \frac{u(h_i^{\frac 1 2}M_{h_i}^{-1}x)}{h_i}, \qquad v_i(y) = \frac{v(h_i^{\frac 1 2}M_{h_i}y)}{h_i},\]
        satisfy 
        \[(\nabla u_i)_{\sharp}(g(h_i^{\frac 1 2}M_{h_i}^{-1}x)dx) = g'(h_i^{\frac 1 2}M_{h_i}y)dy.\]
        Moreover by Proposition~\ref{prop: C1-delta-estimate}, we have $\|M_h\|+\|M_h^{-1}\|\leq h^{-\frac 1 2+\delta}$ for some $\delta>0$, and hence by the continuity of $g(x)$ and $g'(y)$, we know that 
        \[g_i(x):= g(h_i^{\frac 1 2}M_{h_i}^{-1}x)\to g(0)\] 
        and 
        \[g'_i(y):= g'(h_i^{\frac 1 2}M_{h_i}y)\to g'(0)\] as $h_i\to 0$. 
        Therefore after passing to a subsequence, we set $0<c := g(0)$ and $0<c' := g'(0)$, then it follows by Proposition~\ref{prop: C1-delta-estimate} that a subsequence of $(\bar u_i, \bar v_i)$ converges in $C^{1, \alpha}_{loc}$ to a limiting function $(\bar u_{\infty}, \bar v_{\infty})$, which gives an optimal transport map $((\mathtt C, cdx), (\mathtt C', c'dy), u_{\infty}, v_{\infty})$. Moreover, we have
        \[\lim_{i\to \infty}h_i^{-\frac n 2}\int_{D_{h_i^{\frac 1 2}r}(u, 0)}g(x)dx =\lim_{i\to \infty}\int_{D_r(u_i, 0)}g_i(x)dx = g(0)\int_{D_r(u_{\infty}, 0)}dx\]
        hence the volume ratio $\phi_{\infty}(r) = r^{-n}|D_r(u_{\infty}, 0)|$ of this blow-up is constant and equal to $\lim_{r\to 0}g(0)^{-1}e^{A\epsilon (r)}r^{-n}\mu(D_r(u, 0))>0$, where 
        \[
        \epsilon(r) = \frac{1}{\delta}\int_{0}^{r^{\delta}} \frac{\theta_g(s)}{s} +\frac{\theta_{g'}(s)}{s}  ds
        \]
        for $\delta$ depending on $C, n$ and this quantity is $o(1)$ as $r\rightarrow 0$.
        It follows from rigidity case of the monotonicity formula (Theorem~\ref{thm: monotonicity}) that $(\mathtt C, \mathtt C')$ are cones and $(u_{\infty}, v_{\infty})$ are homogeneous of degree 2. 
        \end{proof}

        \begin{remark}
            We note that if $((\mathtt C,cdx), (\mathtt C', c'dy), u_{\infty}, v_{\infty})$ is a blow-up of the optimal transport map $((\Omega, g(x)dx), (\Omega', g'(y)dy), u, v)$ at $0\in \partial\Omega$, it is not necessarily the case that the cones $\mathtt C$ and $\mathtt C'$ are affine equivalent to the tangent cones of $\Omega$ and $\Omega'$ at $0$. This is because the sequence of affine transformations $A_h$ may degenerate and change the shape of the convex sets in the limit. However, this is true if the sections of $u$ are ``round" (see Definition~\ref{def: round}).
        \end{remark}
        \subsection{With homogeneous density}
        In this section, we consider the case where $g(x)$ and $g'(y)$ are homogeneous of degree $l$ and $k$ respectively.  
        \begin{theorem}[Homogeneity of blow-ups]\label{thm: homog-of-blowup-density}
        Let $((\Omega, g(x)dx), (\Omega', g'(y)dy), u, v)$ be an optimal transport map with $C^{\alpha}$ densities $(g(x), g'(y))$ which are homogeneous of degree $(l, k)$ for $k, l\geq 0$. Suppose that the sections $S_h^c(u, 0)$ satisfy
        \[B_{C^{-1}h^{\frac{1}{1+\frac{n+l}{n+k}}}}(0)\subset S_h^c(u, 0)\subset B_{Ch^{\frac{1}{1+\frac{n+l}{n+k}}}}(0).\]
        Then for any sequence of scales $h_i\to 0$, there exists a subsequence along which a blow-up exists, and is of the form $((\mathtt C, g(x)dx), (\mathtt C', g'(y)dy), u_{\infty}, v_{\infty})$. Moreover $(\mathtt C, \mathtt C')$ are the tangent cones of $(\Omega, \Omega')$ at $0$, and $(u_{\infty}, v_{\infty})$ are homogeneous of degree $1+\frac{n+l}{n+k}$ and $1+\frac{n+k}{n+l}$ respectively. 
        \end{theorem}

        \begin{proof}
            By the assumption on the roundeness of the section we can take $A_h = h^{-\frac{1}{1+\frac{n+l}{n+k}}}Id$ when normalizing $S_h^c(u, 0)$. The rescaled functions are given by
            \[u_i(x) = \frac{u(h_i^{\frac{1}{1+\frac{n+l}{n+k}}}x)}{h_i}, \qquad v_i(y) = \frac{v(h_i^{\frac{1}{1+\frac{n+k}{n+l}}}y)}{h_i},\]
            and it follows that $((h_i^{-\frac{1}{1+\frac{n+l}{n+k}}}\Omega, g(x)dx), (h_i^{-\frac{1}{1+\frac{n+k}{n+l}}}\Omega', g'(y)dy), u_i, v_i)$ are a sequence of optimal transport maps. By our rescaling, we have $(h_i^{-\frac{1}{1+\frac{n+l}{n+k}}}\Omega, h_i^{-\frac{1}{1+\frac{n+k}{n+l}}}\Omega')\to (\mathtt C, \mathtt C')$ where $(\mathtt C, \mathtt C')$ is the tangent cone of $(\Omega, \Omega')$ at $0$. By Proposition~\ref{prop: C1-delta-estimate}, after taking a subsequence, $(\bar u_i, \bar v_i)$ will converge to some limiting optimal tranpsort map $ (\bar u_{\infty}, \bar v_{\infty})$ in $C^{1, \alpha}_{loc}$. Moreover, the equality in the monotonicity formula is achieved for the limit, which implies that $(u_{\infty}, v_{\infty})$ are homogeneous of degree $1+\frac{n+l}{n+k}$ and $1+\frac{n+k}{n+l}$ respectively. 
        \end{proof}

        \begin{remark}
            Notice that the existence of a homogeneous density breaks the affine invariance of the problem, which is why we need the additional assumption on the shape of the sections to extract a homogeneous blow-up. 
        \end{remark}

\section{Global regularity for non-degenerate densities}\label{sec: C-1-1-eps-regularity}
\subsection{$C^{1, 1-\eps}$ and $W^{2, p}$ regularity}

	In this section, we prove that an optimal transport map between convex domains equipped with non-degenerate Dini continuous densities is globally $C^{1, 1-\eps}$ for any $\eps>0$ and $W^{2, p}$ for any $p>1$. This result is optimal, as $C^{1, 1}$-regularity does not always hold without any additional regularity assumption on the boundary. The case when $n = 2$ and $g = g' = 1$ was previously proved by Savin-Yu \cite{Savin-Yu}.

	\begin{prop}\label{prop: C-1-alpha bound on sections}
			Suppose $((\Omega, g(x)dx), (\Omega', g'(y)dy), u, v)$ is an optimal transport map with $0\in \overline{\Omega}$, and assume $g(x), g'(y)$ are positive and Dini continuous in $\overline{\Omega}, \overline{\Omega'}$ respectively. Then for any $\epsilon>0$, there exist $h_0>0$ such that for all $0<h<h_0$, we have
			\[(1-\epsilon)(\frac{1}{2})^{\frac{1}{2}}S^c_h(u, 0)\subseteq S^c_\frac{h}{2}(u, 0)\subseteq (1+\epsilon)(\frac{1}{2})^{\frac{1}{2}}S^c_h(u, 0).\]	\end{prop}
	\begin{proof}
		Since the two inclusions follow from exactly the same argument, we will only prove the second inclusion. Suppose it is not true, then there exists a sequence $h_i\to 0$ such that 
			\[S^c_\frac{h_i}{2}(u, 0)\nsubseteq (1+\epsilon)(\frac{1}{2})^{\frac{1}{2}}S^c_{h_i}(u, 0).\]
	It follows that 
			\[A_{h_i}(S^c_\frac{h_i}{2}( u, 0))\nsubseteq (1+\epsilon)(\frac{1}{2})^{\frac{1}{2}}A_{h_i}(S^c_{h_i}( u, 0)).\]
            But by Theorem~\ref{thm: homog-of-blowup-lebesgue-measure}, we know that after taking a subsequence we can extract a homogeneous blow-up limit $((\mathtt C, cdx), (\mathtt C', c'dy), u_{\infty}, v_{\infty})$, hence we have the convergence of centered sections $A_{h_i}(S^c_\frac{h_i}{2}(u, 0))\to S_{\frac 1 2}^c( u_{\infty}, 0)$ and $A_{h_i}(S^c_{h_i}( u, 0))\to S_1^c(u_{\infty}, 0)$, which implies
            \[S_{\frac 1 2}^c( u_{\infty}, 0) \nsubseteq (1+\frac{\epsilon}{2})(\frac 1 2)^{\frac 1 2}S_1^c( u_{\infty}, 0).\]
            But this is a contradiction because $u_{\infty}$ is homogeneous of degree 2 and hence $S_{\frac 1 2}^c(u_{\infty}, 0) = (\frac 1 2)^{\frac 1 2}S_1^c(u_{\infty}, 0)$. 
	\end{proof}

        \begin{corr}\label{corr: pointwise-C-1-alpha-estimate}
            Given any optimal transport map $((\Omega, g(x)dx), (\Omega', g'(y)dy), u, v)$ with $0\in \overline{\Omega}$ and $g(x), g'(y)$ are positive and Dini continuous in $\overline{\Omega}, \overline{\Omega'}$. Then for any $\eps>0$, there exist $h_0>0$ depending on $n, \eps$, and $((\Omega, g(x)dx), (\Omega', g'(y)dy), u, v)$ such that for all $0<h<h_0$, we have
            \[h^{\frac 1 2+\eps}S_1^c(u, 0)\subset S^c_h(u, 0)\subset h^{\frac 1 2-\eps}S_1^c(u, 0).\]
        \end{corr}
        \begin{proof}
            This follows by iteratively applying Proposition~\ref{prop: C-1-alpha bound on sections} to $u$ for $\epsilon$ small enough. 
        \end{proof}

        Now we make this bound effective. 
        \begin{prop}\label{prop: effective-almost-C-1-1-eps}
            Let $((\Omega, g(x)dx), (\Omega', g'(y)dy), u, v)$ be an optimal transport map where the densities $(g(x), g'(y))$ satisfy
	    \[C^{-1}\leq g(x), g'(y)\leq C\]
    	   and
	    \[[g]_{{\rm Dini}, \overline{\Omega}}+     [g']_{{\rm Dini}, \overline{\Omega'}}\leq C\]
            for some $C>1$. Suppose $0\in\overline\Omega\cap\overline{\Omega'}$, and $0 = u(0) = |\nabla u|(0)$, and there exist $R>1$ such that 
            \begin{equation}
                B_{R^{-1}}(0)\subset S^c_1(u, 0)\subset B_{R}(0). 
            \end{equation}
            Then for any $\eps>0$, there exist $0<h_0<\frac 1 2$ depending only on $n$, $C$, $\eps$, and $R$ such that we have
            \[h^{\frac{1}{2}+\eps}S_1^c(u, 0)\subset S^c_h(u, 0)\subset h^{\frac 1 2-\eps}S_1^c(u, 0)\]
            for some $h\in [h_0, \frac 1 2]$. 
        \end{prop}
        \begin{proof}
            \sloppy Suppose this is not true, then there exist a sequence of optimal transport maps $((\Omega_i, g_i(x)dx), (\Omega_i', g'_i(y)dy), u_i, v_i)$ satisfying the stated bounds such that given any $h\in (0, \frac 1 2]$, for $i$ sufficiently large, we have either 
            \[h^{\frac 1 2 +\eps}S_1^c(u, 0)\nsubseteq S^c_h(u_i, 0) \text{ or } S^c_h(u_i, 0) \nsubseteq h^{\frac 1 2-\eps}S_1^c(u, 0).\]
            However, this sequence of optimal transport maps is compact. Indeed by Lemma~\ref{lem: uniform density}, $(\Omega_i, \Omega'_i)$ converges up to a subsequence in the locally Hausdorff sense to some $(\Omega_{\infty}, \Omega'_{\infty})$. It follows from Lemma~\ref{lem: Dini-compactness} that the densities $(g_i(x), g'_i(y))$ converge uniformly along a subsequence to $(g_{\infty}(x), g'_{\infty}(y))$, which are positive and Dini continuous, and satisfy the bound
            \[
            [g_{\infty}]_{{\rm Dini},\Omega_{\infty}} +[g'_{\infty}]_{{\rm Dini},\Omega'_{\infty}} \leq C.
            \]
            Moreover, by Proposition~\ref{prop: C1-delta-estimate} the potentials $(\overline u_i, \overline v_i)$ converge in $C^{1, \gamma}_{loc}$ to a limiting potential $(\overline u_{\infty}, \overline v_{\infty})$, which gives an optimal transport map {$((\Omega_{\infty}, g_{\infty}(x)dx), (\Omega_{\infty}', g'_{\infty}(y)dy), u_{\infty}, v_{\infty})$}. It follows from Corollary~\ref{corr: pointwise-C-1-alpha-estimate} that for some $0<\eps'<\eps$, there exist some $0<h<\frac 1 2$ for which  
            \[h^{\frac 1 2+\eps'}S_1^c(u_{\infty}, 0)\subset S^c_h(u_{\infty}, 0)\subset h^{\frac 1 2 -\eps'}S_1^c(u_{\infty}, 0). \]           
            But this contradicts our assumption since we have 
            \[A_iS_h^c(u_i, 0)\to S_h^c(u_{\infty}, 0)\]
            and 
            \[A_iS_1^c(u_i, 0)\to S_1^c(u_{\infty}, 0).\]
        \end{proof}

        \begin{prop}\label{prop: effective-C-1-1-eps}
            Let $((\Omega, g(x)dx), (\Omega', g'(y)dy), u, v)$ be an optimal transport map where the densities $(g(x), g'(y))$ satisfy
	    \[C^{-1}\leq g(x), g'(y)\leq C\]
    	   and
	    \[[g]_{{\rm Dini}, \overline{\Omega}}+[g']_{{\rm Dini}, \overline{\Omega'}}\leq C\]
            for some $C>1$. Suppose $0\in\overline\Omega\cap\overline{\Omega'}$, and $0 = u(0) = |\nabla u|(0)$, and there exist $R>1$ such that 
            \begin{equation}
                B_{R^{-1}}(0)\subset S^c_1(u, 0)\subset B_{R}(0). 
            \end{equation}
            Then for any $\eps>0$, there exist $M>1$ depending only on $n, C$, and $R$ such that
            \[M^{-1}h^{\frac 1 2+\eps}S_1^c(u, 0)\subset S^c_h(u, 0)\subset Mh^{\frac 1 2 -\eps}S_1^c(u, 0)\]
            for all $h\leq 1$. 
        \end{prop}
        \begin{proof}
            By the uniform density estimate (Lemma~\ref{lem: uniform density}), for any sequence $h_i\to 0$ and a sequence of normalization matrices $A_{h_i} = h_i^{\frac 1 2}M_{h_i}$ with $\det M_{h_i} = 1$, the rescaled optimal transport maps $((\Omega_i, g_i(x)dx), (\Omega'_i, g'_i(y)dy), u_i, v_i)$ with $g_i(x) = g(h_i^{\frac 1 2}M_{h_i}^{-1}x)$, $g'_i(y) = g'(h_i^{\frac 1 2}M_{h_i}y)$ satisfy the hypothesis of Proposition~\ref{prop: effective-almost-C-1-1-eps} uniformly for some new constants $C, R$. Therefore, we can iterate Proposition~\ref{prop: effective-almost-C-1-1-eps} to get the result. 
        \end{proof}
        
	From this, we immediately obtain the global $C^{1, 1-\epsilon}$ regularity of optimal transport maps. 
	\begin{theorem}[Global $C^{1, 1-\epsilon}$ regularity]\label{thm: global-C-1-1-eps-regularity}
		Let $((\Omega, g(x)dx), (\Omega', g'(y)dy), u, v)$ be an optimal transport map where the densities $(g(x), g'(y))$ satisfy
	    \[C^{-1}\leq g(x), g'(y)\leq C\]
    	   and
	    \[[g]_{{\rm Dini}, \overline{\Omega}}+[g']_{{\rm Dini}, \overline{\Omega'}}\leq C\]
            for some $C>1$. Then for any $\eps>0$, we have $u\in C^{1, 1-\eps}(\overline\Omega)$ and $v\in C^{1, 1-\eps}(\overline{\Omega'})$ and moreover, we have the estimate
        \[\|\nabla u\|_{C^{1-\eps}(\overline\Omega)}+\|\nabla v\|_{C^{1-\eps}(\overline{\Omega'})}\leq M\]
        where $M$ depends only on $n, C$, $\eps$, and the inner and outer radius of $\Omega$ and $\Omega'$. 
	\end{theorem}
	\begin{proof}
            Since the inner and outer radius is controlled, let us denote them by $r$ and $R$, so we have $\Omega, \Omega'\subset B_R(0)$ and $B_r(x_0)\subset \Omega$. Then from the outer radius bound for $\Omega'$, we know that $|\nabla u|\leq R$. Moreover, the inner radius bound implies that for some $h_0>R^2\geq \text{osc}_{\overline\Omega} u$, we have $S_{h_0}(u, x)\supset B_r(x_0)$ for all $x\in\overline\Omega$, which implies
            \[|S_{h_0}(u, x)|\geq cr^n>0.\]
            Therefore by Proposition~\ref{prop: effective-C-1-1-eps}, we have $\|\nabla u\|_{C^{1-\eps}(\overline{\Omega})}\leq C$. The same argument applies to $v$. 
	\end{proof}
        The argument of \cite[Theorem 1.2]{Chen-Liu-Wang}, \cite[Theorem 1.1]{Savin-Yu} can be applied to also give $W^{2, p}$-estimates. 
	\begin{theorem}[Global $W^{2, p}$ regularity]\label{thm: global-W-2-p-regularity}
		Let $((\Omega, g(x)dx), (\Omega', g'(y)dy), u, v)$ be an optimal transport map where the densities $(g(x), g'(y))$ satisfy
	    \[C^{-1}\leq g(x), g'(y)\leq C\]
    	   and
	    \[[g]_{{\rm Dini}, \overline{\Omega}}+[g']_{{\rm Dini}, \overline{\Omega'}}\leq C\]
            for some $C>1$. Then for any $p>1$, $u\in W^{2, p}(\overline\Omega)$ and $v\in W^{2, p}(\overline{\Omega'})$, and moreover, we have the estimate
        \[\|\nabla u\|_{W^{1, p}(\overline\Omega)}+\|\nabla v\|_{W^{1, p}(\overline{\Omega'})}\leq M\]
        where $M$ depends only on $n, C$, $p$, and the inner and outer radius of $\Omega$ and $\Omega'$. 
	\end{theorem}
        \begin{proof}
            This follows from the argument of  \cite[Theorem 1.2]{Chen-Liu-Wang} (see also \cite[Theorem 1.1]{Savin-Yu}), making use of Savin's Vitali covering technique \cite{Savin} and Proposition~\ref{prop: effective-C-1-1-eps}. 
        \end{proof}

\subsection{$C^{2, \alpha}$-regularity for domains with $C^{1, \alpha}$-boundary}
        In this section, we show that if $\Omega, \Omega'$ have $C^{1, \beta}$-boundary, and $g(x), g'(y)$ are $C^{\alpha}$, then in fact the solution to optimal transport problem is $C^{2, \min(\alpha, \beta)}$. This is sharp and improves previous results of Caffarelli \cite{Caffarelli3} and Chen-Liu-Wang \cite{Chen-Liu-Wang}, which required the domains to have $C^{1, 1}$-boundary. 

        First, we prove a lemma which says that if the boundary is $C^{1, \alpha}$, then any blow-up has to live on a half-space. 
        \begin{lem}\label{lem: blow-up-is-half-space}
	Let $((\Omega, g(x)dx), (\Omega', g'(y)dy), u, v)$ be an optimal transport map where the densities $g(x), g'(y)$ are positive and Dini continuous in $\overline\Omega, \overline{\Omega'}$ respectively. 
    Suppose $0\in\partial\Omega\cap\partial\Omega'$ with $u(0) = |\nabla u|(0) = 0$,  $\Omega\subset \{x_n\geq 0\}$ and locally near $0$,
    \[\overline\Omega\cap B_1 =  \{(x', x_n)\in \mathbb R^{n-1}\times \mathbb R: x_n\geq \phi(x')\}\cap B_1\]
    where $\phi:B_1^{n-1}\to \mathbb R$ is convex and for all $x'\in B_{1}^{n-1}$, we have
	\[0\leq\phi(x')\leq C|x'|^{1+\alpha}.\]
	Given $h_i\to 0$, let $A_{h_i}$ be the sequence of normalizing matrices for $S_{h_i}^c(u, 0)$, then for any $R>1$, one has
	\[\lim_{h_i\to 0}\frac{|A_{h_i}\Omega\cap B_R|}{|B_R|} \to \frac 1 2.\]
    In particular, the domains $A_{h_i}\Omega$ and $A_{h_i}\{x_n> 0\}$ both converge to some half-space $H$. 
    \end{lem}
    \begin{proof}
        Without loss of generality, we can assume $A_h = h^{-\frac 1 2}M_h$ with $\det M_h=1$ and from Proposition~\ref{prop: effective-C-1-1-eps} we also have $\|M_h\|\ll h^{-\eps}$ for any $\eps>0$. Therefore, we have
	\[A_h\Omega = M_h\{x_n\geq h^{-\frac 1 2}\phi(h^{\frac 1 2}x')\}\]
	and we have
	\[|A_h\Omega\cap B_R| = |\{x_n\geq h^{-\frac 1 2}\phi(h^{\frac 1 2}x')\}\cap M_h^{-1}B_R|\supset |\{x_n\geq Ch^{\frac{\alpha}{2}}|x'|^{1+\alpha}\}\cap M_h^{-1}B_R|.\]
    We can bound $|A_h\Omega\cap B_R|$ as follows
    \begin{align}
    \frac{1}{2}|B_R|&\geq |A_h\Omega\cap B_R|\\
    &\geq |\{x_n\geq Ch^{\frac{\alpha}{2}}|x'|^{1+\alpha}\}\cap M_h^{-1}B_R|\\
    &= |\{x_n\geq 0\}\cap M_h^{-1}B_R|-|\{0\leq x_n\leq Ch^{\frac{\alpha}{2}}|x'|^{1+\alpha}\}\cap M_h^{-1}B_R|\\
    &\geq \frac 1 2 |B_R|- |\{0\leq x_n\leq Ch^{\frac{\alpha}{2}}|x'|^{1+\alpha}\}\cap B_{h^{-\eps}R}|\\
    & \geq \frac 1 2 |B_R|-Ch^{\frac{\alpha}{2}}\int_{B_{h^{-\eps}R}^{n-1}}|x'|^{1+\alpha}dx'\\
    & \geq  \frac 1 2 |B_R|-C(n, \alpha, R)h^{\frac{\alpha}{2}-\eps(n+\alpha)}
    \end{align}
    where from third line to fourth line, we used $\|M_h\|\ll h^{-\eps}$. It follows that if $\eps$ is sufficiently small, then as $h\to 0$, we have
    \[|A_h\Omega\cap B_R|\to \frac 1 2 |B_R|. \]
    It's clear the $A_{h_i}\{x_n\geq 0\}$ must converge to some half-space $H$, and by the volume bound, the $A_{h_i}\Omega$ is then a subset of $H$, but has the same volume, therefore it must also be all of $H$ as well. 
    \end{proof}

    \begin{lem}[Obliqueness]\label{lem: obliqueness}
        Let $((\Omega, g(x)dx), (\Omega', g'(y)dy), u, v)$ be an optimal transport map where the densities $g(x), g'(y)$ are positive and Dini continuous in $\overline\Omega, \overline{\Omega'}$ respectively. Suppose $u(0) = |\nabla u|(0) = 0$, and the domains $\Omega$ and $\Omega'$ are $C^{1, \alpha}$ at $0$ for some $\alpha>0$. If $l_{\Omega}, l_{\Omega'}$ are the defining functions for the supporting hyperplane of $\Omega$ and $\Omega'$ at $0$, then we have
        \[l_{\Omega}\cdot l_{\Omega'}>0.\]
    \end{lem}
    \begin{proof}
        Suppose this is not true, then $l_{\Omega}\cdot l_{\Omega'} = 0$. Then we consider a blow up of \\$((\Omega, g(x)dx), (\Omega', g'(y)dy), u, v)$ at $0$ along some subsequence $h_i\to 0$. From Lemma~\ref{lem: blow-up-is-half-space}, we know that
        \[A_{h_i}\Omega\to H\]
        \[h_iA_{h_i}^{-1}\Omega'\to H'\]
        where $H, H'$ are half-spaces defined by defining functions $l_H$ and $l_{H'}$ so that $H = \{x\in \mathbb R^n: l_H(x)\geq 0\}$ and $H' = \{y\in \mathbb R^n: l_{H'}(y)\geq 0\}$, and 
        \[\lim_{h_i\to 0}\frac{l_{\Omega}\circ{A_{h_i}^{-1}}}{|l_{\Omega}\circ{A_{h_i}^{-1}}|}= \frac{l_H}{|l_H|}\]
        and
        \[\lim_{h_i\to 0}\frac{l_{\Omega'}\circ A_{h_i}}{|l_{\Omega'}\circ A_{h_i}|}= \frac{l_{H'}}{|l_{H'}|}.\]
        Moreover, since the blow-up is a homogeneous optimal transport map from $(H, cdx)$ to $(H', c'dy)$, we know that
        \[l_H\cdot l_{H'}>0.\]
        However, we also have 
        \[\frac{l_H\cdot l_{H'}}{|l_H||l_{H'}|} = \lim_{h_i\to 0}\frac{l_{\Omega}\circ{A_{h_i}^{-1}}\cdot l_{\Omega'}\circ A_{h_i}}{|l_{\Omega}\circ{A_{h_i}^{-1}}||l_{\Omega'}\circ A_{h_i}|} =0\]
        which is a contradiction. 
    \end{proof}

    Once the obliqueness is established, the arguments of Chen-Liu-Wang \cite[Section 6]{Chen-Liu-Wang} can be applied to give $C^{2, \alpha}$-estimates. 
    \begin{theorem}\label{thm: global-C-2-a-regularity}
        Suppose $((\Omega, g(x)dx), (\Omega', g'(y)dy), u, v)$ is an optimal transport map where the densities $(g(x), g'(y))$ satisfy
	    \[C^{-1}\leq g(x), g'(y)\leq C\]
    	   and
	    \[[g]_{\alpha, \overline{\Omega}}+[g']_{\alpha, \overline{\Omega'}}\leq C\]
            for some $\alpha\in (0, 1)$ and $C>1$.
            Suppose in addition $\Omega, \Omega'$ have $C^{1, \beta}$-boundary for some $\beta\in(0, 1)$. Then $u\in C^{2, \gamma}(\overline\Omega)$ and $v\in C^{2, \gamma}(\overline{\Omega'})$ for $\gamma = \min(\alpha, \beta)$ and
        \[\|\nabla u\|_{C^{1, \gamma}(\overline{\Omega})}+\|\nabla v\|_{C^{1, \gamma}(\overline{\Omega'})}\leq M\]
        for $M$ that depends only on $n, \alpha, \beta, C$, and the domains $\Omega, \Omega'$. 
    \end{theorem}

    \begin{proof}
        Once we have Lemma~\ref{lem: obliqueness}, we can follow the proof from \cite[Section 6]{Chen-Liu-Wang} with minor modifications. We'll explain the main modification: since $\partial\Omega, \partial\Omega'$ are only $C^{1, \beta}$, Lemma 6.2 of \cite{Chen-Liu-Wang} only holds under the additional assumption $\delta<\frac{\beta}{2}$. Looking through the proof of \cite[Lemma 6.2]{Chen-Liu-Wang}, we see that if we only have $C^{1, \beta}$-boundary, then we only have $a_h\leq h^{\frac{1+\beta}{2}-\eps}$ (instead of $a_h\leq h^{1-\eps}$) and
        \[0\leq D_nu\leq C_1h^{\frac{1+\beta}{2}-\eps}.\]
        Therefore, for the lower barrier, we must use
        \[\check w := (1+h^{\delta})^{\frac 1 n}w-(1+h^{\delta})^{\frac 1 n}h+h+C_1(x_n-Ch^{\frac 1 2 -\eps})h^{\frac{1+\beta}{2}-\eps},\] 
        and the rest of the proof of \cite[Lemma 6.2]{Chen-Liu-Wang} holds assuming $\delta<\frac{\beta}{2}$. 
        The rest of the arguments follow exactly as in \cite[Section 6]{Chen-Liu-Wang}. 
    \end{proof}
    \begin{remark}
        As in \cite[Theorem 6.1]{Chen-Liu-Wang}, one can also obtain the global continuity of $D^2u$ assuming that $\Omega,\Omega'$ have $C^{1,\alpha}$ boundaries for some $\alpha>0$ and that the densities $(g,g')$ are positive and Dini continuous.
    \end{remark}

\section{Roundness of sections}\label{sec: roundness}

In this section, we are interested in studying the shapes of sections, which can be regarded as a weak form of $C^{1, 1}$ regularity.  We make the following definition.

        \begin{defn}\label{def: round}
            Let $((\Omega, g(x)dx), (\Omega', g'(y)dy), u, v)$ be an optimal transport map with homogeneous densities $(g(x), g'(y))$ of degree $(l, k)$. Suppose $0\in\partial\Omega\cap\partial\Omega'$ and $0 = u(0)=|\nabla u|(0)$. We say that the sections of $(u, v)$ at $0$ are {\bf round} if there is a constant $C>1$ and $h_0>0$ such that for all $0<h<h_0$, we have
            \[B_{C^{-1}h^{\frac{1}{1+\frac{n+l}{n+k}}}}(0)\subset S^c_h(u, 0)\subset  B_{Ch^{\frac{1}{1+\frac{n+l}{n+k}}}}(0)\]
            and
            \[B_{C^{-1}h^{\frac{1}{1+\frac{n+k}{n+l}}}}(0)\subset S^c_h(v, 0)\subset  B_{Ch^{\frac{1}{1+\frac{n+k}{n+l}}}}(0).\]
        \end{defn}
        If the sections of $u$ at $0$ are round, and $((\mathtt C, d\mu), (\mathtt C', d\nu), u_{\infty}, v_{\infty})$ is a blow-up of $u$ about $0$, then it is clear that $(\mathtt C, \mathtt C')$ is affine equivalent to the tangent cone of $(\Omega, \Omega')$ about $0$. Thus, the roundness of sections is related to the existence of homogeneous optimal transport maps between the tangent cones of $(\Omega, \Omega')$. Using this idea, we first prove a negative result saying that roundness does not hold if homogeneous optimal transport maps between tangent cones do not exist. 
        \begin{theorem}\label{thm: non-round}
            Let $((\Omega, g(x)dx), (\Omega', g'(y)dy), u, v)$ be an optimal tranport map with homogeneous densities $(g(x), g'(y))$ of degree $(l, k)$, such that $0\in \partial\Omega\cap\partial\Omega'$ and $u(0) = |\nabla u|(0) = 0$. Let $(\mathtt C, \mathtt C')$ be the tangent cones of $(\Omega, \Omega')$ at $0$. If there is no homogenous optimal transport map from $(\mathtt C, g(x)dx)$ to $(\mathtt C', g'(y)dy)$, then the centered sections of $(u, v)$ are not round at $0$. 
        \end{theorem}
        \begin{proof}
            Suppose that the centered sections are round, then $A_h= h^{-\frac{1}{1+\frac{n+l}{n+k}}}Id$ is a family of normalization matrices for $S_h^c(u, 0)$, and by Theorem~\ref{thm: homog-of-blowup-density} the blow-up at $0$ exists and must be homogeneous optimal transport map from $(\mathtt C, g(x)dx)$ to $(\mathtt C', g'(y)dy)$. But this contradicts our hypothesis that no such homogeneous map exist. 
        \end{proof}

        A very natural question is then the following. 
        \begin{ques}\label{ques: roundness}
            Suppose there exists a homogenous optimal transport map between $\mathtt C$ and $\mathtt C'$. Are the sections of $u$ necessarily round at $0$?
        \end{ques}
        We will address this question in subsequent sections. 
        \subsection{Strongly oblique case} Let $(\mathtt C, \mathtt C')$ be the tangent cone of $(\Omega, \Omega')$ at $0$ respectively, in this section, we prove a theorem that guarantees the roundness of sections under a {\em strong obliqueness} condition on the geometry of $(\mathtt C, \mathtt C')$. 
        
        \begin{theorem}\label{thm: strong-oblique-implies-round}
            Let $((\Omega, g(x)dx), (\Omega', g'(y)dy), u, v)$ be an optimal tranport map with homogeneous densities $(g(x), g'(y))$ of degree $(l, k)$, such that $0\in \partial\Omega\cap\partial\Omega'$ and $u(0) = |\nabla u|(0) = 0$. Let $(\mathtt C, \mathtt C')$ be the tangent cones of $(\Omega, \Omega')$ at $0$. If $\overline{\mathtt C'}\setminus \{0\}\subset int(\mathtt C^{\circ})$, then the centered sections of $(u, v)$ are round at $0$. In particular, all blow-ups are affine equivalent to a homogenous optimal transport map between $\mathtt C$ and $\mathtt C'$. 
        \end{theorem}
        \begin{proof}
            First we show the (non-centered) sections $S_h(u, 0)$ and $S_h(v, 0)$ are round. More precisely, we claim there exist $C>1$ depending on the cones $(\mathtt C, \mathtt C')$ such that for all $h<1$, there exist $r(h), r'(h)>0$ such that
            \[B_{C^{-1}r(h)}\cap\Omega\subset S_h(u, 0)\subset B_{Cr(h)}\cap\Omega\]
            and 
            \[B_{C^{-1}r'(h)}\cap\Omega'\subset S_h(v, 0)\subset B_{Cr'(h)}\cap\Omega'. \]
            To see this, we first use the geometric condition $\overline{\mathtt C'}\setminus \{0\}\subset \text{int}(\mathtt C^{\circ})$, which implies there exist $\delta>0$ such that for all $x\in \Omega$ and $y\in \Omega'$,
            \[x\cdot y\geq \delta|x||y|.\]
            Now fix $x_0$ so that $u(x_0) = h$, then by the convexity of $u$, we have
            \[u(x)\geq u(x_0)+\nabla u(x_0)\cdot (x-x_0)\geq h+\delta|\nabla u(x_0)||x|-|\nabla u(x_0)||x_0|\]
            which implies that \[S_h(u, 0)\subset \{\delta|\nabla u(x_0)||x|-|\nabla u(x_0)||x_0|\leq 0\}\cap\Omega = B_{\delta^{-1}|x_0|}(0)\cap \Omega.\] 
            From this we see that if $x_0, x_1$ are two points in $\partial S_h(u, 0)\cap\Omega$, then $|x_1|\leq \delta^{-1}|x_0|$, and the claim follows. 

            Next we show that the centered sections $S_h^c(u, 0)$ and $S_h^c(v, 0)$ are round as well. We note that $B_{C^{-1}r(h)}\cap\Omega\subset S_h(u, 0)\subset B_{Cr(h)}\cap\Omega$ implies that $|S_h(u, 0)|\sim r(h)^n$ and $|S_h(v, 0)|\sim r'(h)^n$, and hence for all $h$ sufficiently small, Proposition~\ref{prop: properties-of-sections} gives
            \[r(h)^nr'(h)^n\sim |S_h(u, 0)||S_h(v, 0)|\leq |S_{Ch}^c(u, 0)||\nabla u(S_{Ch}^c(u, 0))|\sim h^n\]
            which implies
            \[r'(h)r(h)\lesssim h. \]
            Moreover, we also have
            \[ r(h)^{n+l}\lesssim \int_{B_{C^{-1}r(h)}}g(x)dx\leq \int_{S_h(u, 0)}g(x)dx\leq \int_{B_{Cr(h)}}g(x)dx\lesssim r(h)^{n+l}\]
            and 
            \[r'(h)^{n+k}\lesssim \int_{B_{C^{-1}r'(h)}}g'(y)dy\leq \int_{S_h(v, 0)}g'(y)dy\leq \int_{B_{Cr'(h)}}g'(y)dy\lesssim r'(h)^{n+k}, \]
            hence by the monotonicity formula (Theorem~\ref{thm: monotonicity}), we have
            \[r(h)^{n+l}\sim \int_{S_h(u, 0)}g(x)dx\geq \int_{D_{h^{\frac 1 2}}(u, 0)}g(x)dx\gtrsim h^{\frac{n+l}{1+\frac{n+l}{n+k}}}\]
            and
            \[r'(h)^{n+k}\sim \int_{S_h(v, 0)}g'(y)dy\geq \int_{D_{h^{\frac 1 2}}(v, 0)}g'(y)dy\gtrsim h^{\frac{n+k}{1+\frac{n+k}{n+l}}}, \]
            which implies
            \[r(h)r'(h)\gtrsim h. \]
            Hence we have $r(h)\sim h^{\frac{1}{1+\frac{n+l}{n+k}}}$ and $r'(h)\sim  h^{\frac{1}{1+\frac{n+k}{n+l}}}$. This implies $|S_h(u, 0)||S_h(v, 0)|\sim h^n$, which implies by Proposition~\ref{prop: properties-of-sections} that there is a uniform density bound
            \[\frac{|S_h^c(u, 0)\cap\Omega|}{|S_h^c(u, 0)|}, \frac{|S_h^c(v, 0)\cap\Omega'|}{|S_h^c(v, 0)|}\geq \delta >0. \]
            If we let $l_1(h), \ldots, l_n(h)$ denote the lengths of the principle axis of the John ellipsoid of $S_h^c(u, 0)$, then since $B_{cr(h)}(x')\subset S_{C^{-1}h}(u, 0)\subset S_{h}^c(u, 0)$, we have that
            \[l_i(h)\gtrsim r(h) \text{ for all }i = 1, 2, \ldots, n.\]
            Moreover the uniform density bound implies that 
            \[\prod_{i=1}^nl_i(h)\sim |S_h^c(u, 0)|\sim |S_h(u, 0)|\sim r(h)^n.\]
            Altogether, we see that $l_i(h)\sim r(h)$ for all $i = 1, \ldots, n$, which implies that the centered sections of $u$ are round. The same argument applied to $v$ implies that the centered sections of $v$ are round as well. 
        \end{proof}

        \subsection{2D domains with Lebesgue measure}
        In this section, we focus on the case when $n=2$. For simplicity, we also focus on the situation $g = g' =1$. In particular, we'll give a complete answer to Question~\ref{ques: roundness} when $(\Omega, \Omega')$ are planar polytopes.

        In this case, we have the following observation. 
        \begin{prop}\label{prop: quadratic-poly}
            The only quadratic solutions to $\det D^2u =1$ on any cone $\mathtt C\subset \mathbb R^2$ are given by restrictions of a positive definite quadratic polynomial. 
        \end{prop}
        \begin{proof}
            Suppose that $u:\mathtt C\to \mathbb R$ is such a solution, then without loss of generality, we can assume that the cone $\mathtt C$ lies in the upper half space $\{x_2\geq 0\}$. If we set $f(t) := \sqrt{u(t, 1)}$, then $u(x, y) = y^2f(\frac{x}{y})^2$, and this means that $f(t)$ must satisfy the ODE
            \[f''(t) = \frac{c}{f(t)^3},\]
            which can be integrated. First multiply both side by $f'(t)$, which gives
            \[((f'(t))^2)' +c (f(t)^{-2})' = 0. \]
            Integrating this, we have
            \[f'(t)^2+c(f(t))^{-2} = a \implies f'(t) = \sqrt{a-c(f(t))^{-2}}, \]
            for some constant $c$ and $a$. Integrating this again gives
            \[f(t) = \sqrt{A(t-t_0)^2+C}\]
            for some constants $A, t_0, C$, which means 
            \[u(x, y) = A(x-t_0y)^2+Cy^2, \]
            which is a positive definite quadratic polynomial. 
        \end{proof}
        A corollary of this is that we can classify precisely the cones $(\mathtt C, \mathtt C')\subset (\mathbb R^2, \mathbb R^2)$ for which a homogenous optimal transport map of the form $((\mathtt C, dx), (\mathtt C', dy), u, v)$ exists. 
        \begin{corr}\label{corr: 2D-tangent-cones}
            Let $(\mathtt C, \mathtt C')\subset (\mathbb R^2_x, \mathbb R^2_y)$ be two convex cones, then a homogenous optimal transport map of the form $((\mathtt C, dx), (\mathtt C', dy), u, v)$ exists iff we are in one of the following four cases:
            \begin{enumerate}
                \item (Half-space case) Both $\mathtt C$ and $\mathtt C'$ are half-spaces and $l\cdot l'
                 >0$, where $l$ and $l'$ are the linear functions defining the half-space $\mathtt C$ and $\mathtt C'$ respectively. In this case, there is a one-parameter family of optimal maps between $\mathtt C$ and $\mathtt C'$. 
                \item (Acute/strongly oblique case) Both $\mathtt C$ and $\mathtt C'$ are strict cones and $\overline{\mathtt C'}\setminus \{0\}\subset \text{int}(\mathtt C^{\circ})$. In this case, there is a unique homogeneous optimal transport map that maps $\mathtt C$ to $\mathtt C'$. 
                \item (Right angle case) Both $\mathtt C$ and $\mathtt C'$ are strict cones and $\mathtt C' = \mathtt C^{\circ}$, in which case there is a 1-parameter family of homogenous optimal transport maps from $\mathtt C$ to $\mathtt C'$. 
                \item (Obtuse case) Both $\mathtt C$ and $\mathtt C'$ are strict cones and $\overline{\mathtt C^{\circ}}\setminus\{0\}\subset \text{int}(\mathtt C')$. In this case, there is a unique homogeneous optimal transport map that maps $\mathtt C$ to $\mathtt C'$. 
            \end{enumerate}
        \end{corr}
        \begin{proof}[Proof of Corollary]
            
            Assume such a homogenous optimal transport map $((\mathtt C, dx), (\mathtt C', dy), u, v)$ exist, then Proposition~\ref{prop: quadratic-poly} implies $u$ and $v$ are positive quadratic polynomials which are Legendre dual. Hence there exists an affine transformation $A$ with $\det A = 1$ that puts $u$ and $v$ into its standard form, which means
            \[A\cdot ((\mathtt C, dx), (\mathtt C', dy), u, v) = ((\tilde{\mathtt C}, dx), (\tilde{\mathtt C'}, dy), \frac{|x|^2}{2}, \frac{|y|^2}{2}). \]
            In this new form, the optimal transport map is just given by the identity map, hence we have $\tilde{\mathtt C} = \tilde{\mathtt C'}$, and the four cases correspond to different types of cones $\tilde{\mathtt C}$. If $\tilde{\mathtt C}$ is a half-space (which, after a rotation can be assumed to be the upper half-space), that puts us into case 1. In this case, there is a 1-parameter family of optimal transport maps given by
            \[u_{\lambda}(x) = \frac{\lambda x_1^2}{2}+\frac{x_2^2}{2\lambda}. \]
            If $\tilde{\mathtt C}$ is an acute cone (i.e. cone angle $\theta<\frac{\pi}{2}$), then we are in case 2, and in this case the optimal transport map is unique. If $\tilde{\mathtt C}$ has cone angle $\theta = \frac{\pi}{2}$, then we are in case 3. In this case after rotation it can be assumed that $\tilde{\mathtt C} = \{x_1>0, x_2>0\}$) and again there is a 1-parameter family of optimal transport maps given by $u_{\lambda}(x) = \frac{\lambda x_1^2}{2}+\frac{x_2^2}{2\lambda}$. Finally if $\tilde{\mathtt C}$ is an obtuse cone, (i.e. cone angle $\theta>\frac{\pi}{2}$) then we are in case 4, and in this case, the homogeneous optimal transport map is unique as well. 
        \end{proof}

        The next theorem says that if we are in case 4, the sections are always round. 
        \begin{theorem}\label{thm: 2D-case-4-round}
            Let $((\Omega, dx), (\Omega', dy), u, v)$ be an optimal transport map between convex domains in $\mathbb R^2$ satisfying $0\in\partial\Omega\cap\partial\Omega'$ and $0 = u(0) = |\nabla u|(0)$. Let $(\mathtt C, \mathtt C')$ be the tangent cones of $(\Omega, \Omega')$ at $0$. If $(\mathtt C, \mathtt C')$ are both strict cones and $\overline{\mathtt C^{\circ}}\setminus\{0\}\subset \text{int}(\mathtt C')$, then the sections of $(u, v)$ are round at $0$. 
        \end{theorem}
        \begin{proof}            
            First we normalize $(\mathtt C, \mathtt C')$. By applying an affine transformation $A$ with $\det A = 1$, we can arrange that \[\mathtt C = \{x\geq 0, y\geq 0\}\] 
            and 
            \[\mathtt C' = \{y\geq -ax, x\geq -ay\}\] for some $0<a<1$. 
            Moreover since $(\Omega, \Omega')$ is asymptotic to $(\mathtt C, \mathtt C')$ near zero, that means we can write 
            \[\Omega = \{(x, y): y\geq f_+(x), x\geq f_-(y)\}\]
            and
            \[\Omega' = \{(x, y):y\geq-ax+g_+(x), x\geq -ay+g_-(y)\}\]
            where $f_{\pm}, g_{\pm}\geq 0$ are convex functions with $\lim_{t\to 0}\frac{f_{\pm}(t)}{t} = \lim_{t\to 0}\frac{g_{\pm}(t)}{t} = 0$. From here, the basic idea is to show that if the sections are not round, then the direction of the long axis must align with the boundary of the cone $\mathtt C$. By symmetry, this must be true for both $u$ and $v$, but that is inconsistent with the condition $\overline{\mathtt C^{\circ}}\setminus\{0\}\subset \text{int}(\mathtt C')$. 
            
            Let $E_h$ be the John Ellipsoid of $S_h(u, 0)$, which by uniform density (Lemma~\ref{lem: uniform density}) is comparable to $S^c_h(u, 0)$. Suppose for contradiction that $E_h$ is highly eccentric for $h\ll 1$. If we let $v_1(h), v_2(h)$ denote unit vectors in the directions of the long and short axis of $E_h$ respectively, and $l_1(h), l_2(h)$ be the lengths of the long and short axis, then this means we have $v_1\perp v_2$, and $l_1\gg l_2$ for $h\ll 1$. Now we claim that for any $\theta>0$, there exists $h_0>0$ such that for all $0<h<h_0$, the axis vectors must lie $\theta$ close to the two coordinate axis, that is $v_1(h), v_2(h)\in B_{\theta}(e_i)\cup B_{\theta}(e_2)$ for $h$ sufficiently small. To see this we first note that since $l_1(h)\gg l_2(h)$ for $h$ sufficiently small, therefore $\partial S_h(u, 0)\cap \partial\Omega$ is two points, which we call $(x^{\pm}, y^{\pm})$ satisfying $f_{+}(x_+) = y_+$ and $f_-(y_-) = x_-$. We denote 
            \[T_h := \text{Convex Hull}(\{(0, 0), (x_+, y_+), (x_-, y_-)\})\]
            and 
            \[R_h := \{x\geq 0, y\geq 0, y\geq a^{-1}(x-x_+), x\geq a^{-1}(y-y_-) \}, \]
            then we have
            \[T_h\subset S_h(u, 0)\subset R_h.\]
            Moreover without loss of generality, we can assume $x_+\geq y_-$ (otherwise we flip the two coordinate axis), then there exist $c>0$, and $x'\in T_h$ such that $B_{c|y_-|(x')}\subset T_h$, and there exist $C>1$ such that $R_h\subset B_{C|x_+|}(0)$. 
            Without loss of generality, we can assume that $x_+\geq y_-$ (if not, then we can flip the two axis), then for any $C>0$, there is a sufficiently small $h$ for which $x_+\geq Cy_-$. If not, then $x_+\sim y_-$ and
            \[B_{c|x_+|}(x')\subset T_h \subset S_h(u, 0)\subset R_h\subset B_{C|x_+|}(0),\] 
            which would imply that $S_h(u, 0)$ are round. Next we claim for any $\delta>0$, there exist sufficiently small $h$ for which $S_h(u, 0)\subset P_{\delta|x_+|}$, where
            \[P_{\delta|x_+|} = R_h\cap\{y\leq \delta|x_+|\}. \]
            If not, then there exist $(x_0, \delta|x_+|)\in S_h$ with $x_0\leq 2|x_+|$, which means
            \[ \text{ConvexHull}\{(0, 0), (x_+, y_+), (x_0, \delta|x_+|)\}\subset S_h(u, 0),\]
            and if $x_+$ is sufficiently small, we have $y_+<\frac{\delta}{2}x_+$, therefore
            \[B_{c\delta |x_+|}(x')\subset \text{ConvexHull}\{(0, 0), (x_+, y_+), (x_0, \delta|x_+|)\}\subset S_h(u, 0)\]
            which would mean that the sections $S_h(u, 0)$ are round. 
            
            Thus for $h$ sufficiently small
            \[T_h\subset S_h(u, 0)\subset P_{\delta|x_+|}\subset E(C|x_+|, \delta|x_+|).\]
            If we pick $\delta$ sufficiently small, then it follows from this that the long axis of the John ellipsoid of $S_h(u, 0)$ must point $\theta$-close to the $e_1$-direction, which means $v_1\in B_{\theta}(e_1)$. 

            Now we show this is a contradiction. The same argument applied to both $u$ and $v$ shows that we can find sequence $h_i\to 0$ for which $S_{h_i}(u, 0)$ and $S_{h_i}(v, 0)$ are becoming more and more eccentric, and the long axis of $S_{h_i}(u, 0)$ point towards on the coordinate axis, while the long axis of $S_{h_i}(v, 0)$ point towards an axis of $\mathtt C'$, moreover since we know by Lemma~\ref{lem: uniform density} that $|S_h(u, 0)|\sim |S_h(v, 0)|\sim h$, this implies there exist $x_h\in S_h(u, 0)$ pointing in the direction of the long axis (i.e. $\frac{x_h}{|x_h|}\in B_{\theta}(e_i)$) with $|x
            _h|\gg h^{\frac 1 2}$, and there exist $y_h\in S_h(v, 0)$ pointing in the direction of the long axis (i.e. $\frac{y_h}{|y_h|}\in B_{\theta}((\frac{1}{\sqrt{1+a^2}}, -\frac{a}{\sqrt{1+a^2}}))$) such that $|y_h|\gg h^{\frac 1 2}$. Since the axis of $\mathtt C$ and $\mathtt C'$ are not perpendicular, it follows that for $h$ sufficiently small
            \[|x_h\cdot y_h| \geq \delta |x_h||y_h|\gg h,\]
            but this contradicts Proposition~\ref{prop: properties-of-sections}. 
        \end{proof}
        As a consequence of this, we can give a complete answer to Question~\ref{ques: roundness} for polygonal domains in the plane. 
        \begin{theorem}\label{thm: planar-C-1-1-reg-characterization}
            Let $((\Omega, dx), (\Omega', dy), u, v)$ be an optimal transport map between two polygonal domains in $\mathbb R^2$. Suppose $0\in\partial\Omega\cap\partial\Omega'$ with $0 = u(0)= |\nabla u|(0)$ and let $(\mathtt C, \mathtt C')$ be the tangent cones of $(\Omega, \Omega')$ at $0$. Then the sections at $0$ are round iff there is a homogeneous optimal transport map between $(\mathtt C, \mathtt C')$. (i.e. iff we are in one of the situations of Corollary~\ref{corr: 2D-tangent-cones}.)
        \end{theorem}
        \begin{proof}
            If there is no homogeneous optimal transport map between the tangent cones, by Theorem~\ref{thm: non-round}, the sections cannot be round at $0$. On the other hand, if there is a homogeneous optimal transport map between the tangent cones $(\mathtt C, \mathtt C')$, then we are in one of the four situations of Corollary~\ref{corr: 2D-tangent-cones}, so we consider the 4 cases separately. In the first case, we can reflect the optimal transport map across the boundary of the half-space, which turns the origin into an interior point, which means the sections are round. In the second case, Theorem~\ref{thm: strong-oblique-implies-round} tells us that the sections are always round. In the third case, we can reflect the optimal transport map across one of the boundary sides to return to case 1. In the fourth case, Theorem~\ref{thm: 2D-case-4-round} tells us that the sections are always round. 
        \end{proof}

        \begin{remark}
            In the case when the domains $(\Omega, \Omega')$ are not polygonal, if we are in cases 2 and 4 of Corollary~\ref{corr: 2D-tangent-cones}, then Theorem~\ref{thm: strong-oblique-implies-round} and Theorem~\ref{thm: 2D-case-4-round} still apply. The only cases where roundness is not known are cases 1 and 3. In those cases, any blow-up must live on the tangent cone, but there is a one-parameter family of homogeneous optimal transport maps between the tangent cones, and it is conceivable that there could be a solution that interpolates between different members of this family at different scales. 
        \end{remark}
	\subsection{Homogeneous densities}
        In this section, we focus on the situation when the domain $\Omega\subset \{x_n\geq 0\}$ is locally a strict cone, and the target is locally equal to the half-space $\{y_n\geq 0\}$ equipped with a homogenous density $g'(y) = y_n^k$. This situation arises in the study of complete Calabi-Yau metrics of Tian-Yau type in our previous work with S.-T. Yau \cite{Collins-Tong-Yau}, and as a limiting description of intermediate complex structure degeneration of Calabi-Yau manifolds, as shown in the work of Li \cite{Li-intermediate}. In fact, together with S.-T. Yau, we have constructed homogeneous optimal transport maps between the cones of this type \cite[Corollary 1.1]{Collins-Tong-Yau}. In this situation, we show that the roundness of sections always hold. 
        \begin{theorem}\label{thm: round-degenerate-density}
            Suppose that near the origin, $\Omega$ is equal to a strict cone $\mathtt C = \{x_n\geq \phi_{\mathtt C}(x')\}$, so $\Omega\cap B_{r_0} =\mathtt C\cap B_{r_0}$ and $\Omega'$ is locally a halfspace $\Omega'\cap B_{r_0} = \{y_n\geq 0\}\cap B_{r_0}$, and moreover assume the density $g(x)$ which is homogeneous of degree $l$ and $g'(y) = y_n^k$. If in addition $k> l\geq 0$, then the sections of $u$ are round. In particular, there exist $C>1$, and $h_0>0$ such that for all $0<h<h_0$, we have
            \[B_{C^{-1}h^{\frac{1}{1+\alpha}}} \subset S^c_h(u, 0)\subset B_{Ch^{\frac{1}{1+\alpha}}}\]
            and 
            \[B_{C^{-1}h^{\frac{1}{1+\frac{1}{\alpha}}}}\subset S_h^c(v, 0)\subset B_{Ch^{\frac{1}{1+\frac{1}{\alpha}}}},\]
            where $\alpha = \frac{n+l}{n+k}\in (0, 1]$. 
        \end{theorem}
    \begin{proof}
        Let us $E_h$ be the John ellipsoid of the section $S^c_h(v)$, then $E_h\cap\{y_n = 0\}$ is also an ellipsoid centered at $0$ in $\mathbb R^{n-1}$, let $l_1(h)\leq \cdots\leq l_{n-1}(h)$ be the lengths of the principle axis of this ellipsoid, which after a rotation by $A\in SO(n-1)$, we can assume is in the direction of the coordinate axis. We also denote $d(h): = \sup\{x_n: x\in S_h^c(v, 0)\}$ to be the height of $E_h$, then there exist $C>1$ such that
        \[C^{-1}E(l_1(h), \ldots, l_{n-1}(h), d(h))\subset S_h^c(v, 0)\subset CE(l_1(h), \ldots, l_{n-1}(h), d(h)). \]
        By Proposition~\ref{prop: properties-of-sections}, we also have
        \[C^{-1}E\left(\frac{h}{l_1(h)}, \ldots, \frac{h}{l_{n-1}(h)}, \frac{h}{d(h)}\right)\subset S_h^c(u, 0)\subset CE\left(\frac{h}{l_1(h)}, \ldots, \frac{h}{l_{n-1}(h)}, \frac{h}{d(h)}\right)\]
        for $C>1$ depending additionally on the doubling constant of $g$. 
        
        Since $\Omega'$ is locally a half-space, it follows that the centered sections of $v$ have a uniform lower density bound
        \[\frac{|S_h^c(v, 0)\cap\Omega'|}{|S_h^c(v, 0)|}\geq c>0\]
        and by Proposition~\ref{prop: properties-of-sections}, the centered sections of $u$ also has a similar density bound
        \[\frac{|S_h^c(v, 0)\cap\Omega'|}{|S_h^c(v, 0)|}\geq c>0.\]
        Moreover, since $\mathtt C$ is a strict cone, we have 
        \[\mathtt C\subset \mathtt C(B_R(0)) := \{(x', x_n)\in \mathbb R^n: x_n\geq \frac 1 R|x'|\}\]
        for some $R>1$. This implies the supporting hyperplanes of the sections $S_h(v, 0)$ at points $y\in S_h(v, 0)\cap\{y_n = 0\}$ have slope at most $R$, which implies
        \[d(h)\leq  CR l_i(h) \text{ for }i = 1, \ldots, n-1. \]
        Moreover, by Proposition~\ref{prop: properties-of-sections}, 
        \[\nu(S_h(v, 0))\sim \int_{E(l_1(h), \cdots, l_{n-1}(h), d(h))\cap\{y_n\geq 0\}}y_n^kdy\sim l_1(h)\cdots l_{n-1}(h)d(h)^{k+1}\]
        and 
        \[\mu(S_h(u, 0))\sim \int_{E(\frac{h}{l_1(h)}, \cdots, \frac{h}{l_{n-1}(h)}, \frac{h}{d(h)})\cap \mathtt C}g(x)dx\sim \frac{h}{l_1(h)}\cdots \frac{h}{l_{n-1}(h)}(\frac{h}{d(h)})^{l+1}. \]
        Therefore, the equation tells us that
        \[l_1(h)\cdots l_{n-1}(h)d(h)^{k+1}\sim\frac{h}{l_1(h)}\cdots \frac{h}{l_{n-1}(h)}(\frac{h}{d(h)})^{l+1} \implies (l_1(h)\cdots l_{n-1}(h))^2d(h)^{k+l+2}\sim h^{n+l}. \]
        This implies that 
        \[d(h)^{2n+k+l}\lesssim h^{n+l} \implies d(h)\lesssim h^{\frac{n+l}{2n+k+l}}. \]
        But the monotonicity formula tells us that 
        \[\mu(S_h(v, 0)) \sim l_1(h)\cdots l_{n-1}(h)d(h)^{k+1}\gtrsim h^{\frac{n+l}{1+\frac{n+l}{n+k}}}\implies d(h)^{\frac{k-l}{2}}\gtrsim h^{\frac{(n+l)(k-l)}{2(2n+k+l)}}\]
        since $k\geq l\geq 0$, this implies
        \[d(h)\gtrsim h^{\frac{n+l}{2n+k+l}}. \]
        Therefore we have $d(h)\sim l_i(h)\sim h^{\frac{n+l}{2n+k+l}}$, and the proposition is proved. 
\end{proof}

\begin{corr}
    Under the assumptions of Theorem~\ref{thm: round-degenerate-density}, blow-ups exists and are homogeneous optimal transport maps between the corresponding tangent cones. 
\end{corr}

\begin{proof}
This follows immediately from Theorem~\ref{thm: homog-of-blowup-density}.
\end{proof}

\begin{remark}
    A consequence of this is that homogeneous solutions of the optimal transport problem constructed in \cite{Collins-Tong-Yau}, which is expected to serve as a model at infinity for ``generalized Tian-Yau metrics", arises as a blow-up limit for the optimal transport problem describing intermediate complex structure degenerations of Calabi-Yau hypersurfaces \cite{Li-intermediate}. This provides further evidence that generalized Tian-Yau metrics should exist and should arise as bubbles in the intermediate complex structure degeneration of Calabi-Yau hypersurfaces. 
\end{remark}

\section{Flat directions and mixed homogeneity}\label{sec: flat-side}
Let $\mathtt C = \{x_n\geq \phi_{\mathtt C}(x')\}\subset \mathbb R^m$ be a strict cone. In this section, we consider optimal transport maps from $(\mathbb R^{n-m}\times \mathtt C, dx)$ to $(\mathbb R^{n-1}\times \mathbb R_+, y_n^kdy)$. This situation arises in our previous work with S.-T. Yau \cite{Collins-Tong-Yau}. We'll see that in this case, the existence of flat-sides in the domain breaks the homogeneity, and the solutions exhibit mixed homogeneous behaviour. 

\subsection{A Pogorelov estimate along flat directions} We first prove Pogorelov-type estimates for solutions in the flat directions.  When $\mathtt{C}= \mathbb{R}_+$ is a one-dimensional cone, a Pogorelov-type estimate was established by Jhaveri-Savin \cite[Proposition 4.1]{Jhaveri-Savin} using a doubling trick.  The following estimate generalizes this result to arbitrary cones and general densities.

\begin{prop}\label{prop: VeryGenPog}
Let $\Omega\subset \mathbb{R}^{n}_{x}$ and $ \Omega' \subset \mathbb{R}_y^n$.  Suppose that the following structural conditions hold: 
\begin{itemize}
    \item[(i)] There is a smooth convex function $\phi_\Omega(x) =\phi_\Omega(x_2,\ldots, x_{n-1})$ such that $\phi_\Omega(0)=0$ and
    \[
    \Omega \cap B_1(0) = \{ x\in B_1(0) : x_n \geq \phi_\Omega(x_2,\ldots x_{n-1})\}
    \]
    \item[(ii)] There is a smooth convex function $\phi_{\Omega'}(y)=\phi_{\Omega'}(y_1,\ldots, y_{n-1})$ such that $\phi_{\Omega'}(0)=0$ and
    \[
    \Omega'\cap B_1(0) =  \{ y\in B_1(0) : y_n \geq \phi_{\Omega'}(y_1,\ldots y_{n-1})\}
    \]
    \item[(iii)] There is a smooth, positive function $g(x)$ defined on $\overline{\Omega\cap B_1(0)} $ and such that $g(x)=g(x_2,\ldots,x_n)$ is independent of $x_1$.
    \item[(iv)]There is a smooth, positive function $g'(y)$ defined on $\overline{\Omega'\cap B_1(0)} $ and such that $g'(y)$ is log-concave and satisfies
    \[
    \sum_{p=1}^{n}\frac{\del g'}{\del y_p}y_p\leq \kappa g'(y)
    \]
    for some $\kappa >0$ and for all $y\in \overline{\Omega'\cap B_1(0)}$.
\end{itemize}
Suppose that $u$ is a convex function satisfying $0 = u(0) = |\nabla u|(0)$ and 
\[
\begin{aligned}
g'(\nabla u)\det D^2u &= g(x),& \quad \text{ in } \Omega \cap B_1(0)\\
\frac{\del u}{\del x_n} &\geq \phi_{\Omega'}(\frac{\del u}{\del x_1}, \ldots, \frac{\del u}{\del x_{n-1}}),& \quad \text{ in } \Omega \cap B_1(0)\\
\frac{\del u}{\del x_n} &= \phi_{\Omega'}(\frac{\del u}{\del x_1}, \ldots, \frac{\del u}{\del x_{n-1}}),& \quad \text{ on } \del\Omega \cap B_1(0)
\end{aligned}
\]
Let $\ell$ be the tangent plane to $u$ at the origin, $h>0$ and suppose that
\[
S_{h}(u,0) = \{u< \ell+h\} \Subset B_1(0)\cap \Omega.
\]
Then we have a bound
\[
u_{11}|u-\ell-h| \leq C(n,\kappa, \|u_{1}-\ell_{1}\|_{L^{\infty}(S_{h}(0))}). 
\]
\end{prop}
\begin{proof}
    By subtracting a linear function from $u$, we may assume that $u(0)= |\nabla u(0)|=0$.  Let $S_h= S_h(u,0)$ be the a section of height $h$ centered at $0$, and replace $u$ with $u-h$.  The argument is a Pogorelov type argument, making essential use of the boundary data.  Consider the quantity
\[
\Phi = \log u_{11} + \log|u| + \frac{1}{2}u_1^2
\]
    Let $x_* \in S_h$ be the point where $\Phi$ achieves its maximum.  Either $x_*$ is an interior point, or $x_*$ lies on $\del\Omega \cap S_h\cap B_1$. 

Suppose first that $x_*$ is an interior point. We may perform a shearing transformation preserving the $x_1$ direction, followed by a rotation in the $x_2,\ldots,x_n$ directions in order to diagonalize $D^2u(x_*)$.  We note that since the $x_1$-direction is preserved, the structural conditions $(i)-(iv)$ are preserved under this affine change of coordinates. Differentiating the equation in the $x_1$-direction and using that $D^2u$ is diagonal at $x_*$  we compute
\begin{equation}\label{eq: LinearizedEquation}
\sum_p \frac{u_{1pp}}{u_{pp}}  + \frac{1}{g'}\frac{\del g'}{\del y_1}u_{11} = 0
\end{equation}
\begin{equation}\label{eq: concavityEquation}
\sum_p \frac{u_{11pp}}{u_{pp}} + \frac{1}{g'}\frac{\del g'}{\del y_p}u_{11p} = \sum_{p,q}\frac{u_{1pq}^2} {u_{pp}u_{qq}}-u_{11}^2\frac{\del^2 \log g'}{\del y_1^2}
\end{equation}

Differentiating $\Phi$ at $x_*$ we have
\begin{equation}\label{eq: firstDerivTestPhi}
\Phi_{p} = \frac{u_{11p}}{u_{11}} + \frac{u_p}{u} + u_{1p}u_1=0
\end{equation}
\begin{equation}\label{eq: secondDerivTestPhi}
 \Phi_{pp} = \frac{1}{u_{11}}u_{11pp}-\frac{(u_{11p})^2}{u_{11}^2} + \frac{u_{pp}}{u} - \frac{u_p^2}{u^2} + u_{1pp}u_1 + u_{1p}^2 \leq0
\end{equation}
Multiplying~\eqref{eq: secondDerivTestPhi} by $u^{pp}= \frac{1}{u_{pp}}$ and summing over $p$ yields
\[
0 \geq \sum_{1\leq p\leq n} \frac{u_{11pp}}{u_{11}u_{pp}}-\sum_{p=1}^{n}\frac{(u_{11p})^2}{u_{11}^2u_{pp}} + \frac{n}{u} - \sum_{p=1}^{n}\frac{u_p^2}{u^2u_{pp}} + u_{1}\sum_{p=1}^{n}\frac{u_{1pp}}{u_{pp}} + u_{11} 
\]
We use~\eqref{eq: firstDerivTestPhi} to replace the fourth term on the right hand side for $p\ne 1$ to get
\[
0 \geq \sum_{1\leq p\leq n} \frac{u_{11pp}}{u_{11}u_{pp}}-\sum_{p=1}^{n}\frac{(u_{11p})^2}{u_{11}^2u_{pp}} + \frac{n}{u} - \sum_{p=2}^{n}\frac{u_{11p}^2}{u_{11}^2u_{pp}}-\frac{u_1^2}{u^2u_{11}} + u_{1}\sum_{p=1}^{n}\frac{u_{1pp}}{u_{pp}} + u_{11}
\]
We now use~\eqref{eq: concavityEquation} to replace the fourth order terms
\begin{equation}\label{eq: PogorelovInequality1}
\begin{aligned}
  0 &\geq \sum_{1 \leq p,q\leq n}\frac{u_{1pq}^2} {u_{pp}u_{qq}u_{11}}-u_{11}\frac{\del^2 \log g'}{\del y_1^2}+\frac{1}{u_{11}}\frac{\del^2\log g}{\del x_1^2} -\sum_{p=1}^{n}\frac{1}{g'u_{11}}\frac{\del g'}{\del y_p}u_{11p} \\
  &\quad -\sum_{p=1}^{n}\frac{(u_{11p})^2}{u_{11}^2u_{pp}} + \frac{n}{u} - \sum_{p=2}^{n}\frac{u_{11p}^2}{u_{11}^2u_{pp}}-\frac{u_1^2}{u^2u_{11}} + u_{1}\sum_{p=1}^{n}\frac{u_{1pp}}{u_{pp}} + u_{11}\\
  &\geq \sum_{2 \leq p,q\leq n}\frac{u_{1pq}^2} {u_{pp}u_{qq}u_{11}}-u_{11}\frac{\del^2 \log g'}{\del y_1^2} -\sum_{p=1}^{n}\frac{1}{g'u_{11}}\frac{\del g'}{\del y_p}u_{11p} \\
  &\quad + \frac{n}{u} -\frac{u_1^2}{u^2u_{11}} + u_{1}\sum_{p=1}^{n}\frac{u_{1pp}}{u_{pp}} + u_{11}
\end{aligned}
\end{equation}
We now use~\eqref{eq: firstDerivTestPhi} to obtain
\[
\begin{aligned}
-\sum_{p=1}^{n}\frac{1}{g'u_{11}}\frac{\del g'}{\del y_p}u_{11p} &= \sum_{p=1}^{n}\frac{1}{g'}\frac{\del g'}{\del y_p}\left(\frac{u_{p}}{u}+u_{1p}u_1\right)\\
&= \frac{1}{g'}\frac{\del g'}{\del y_1}u_{11}u_1 +\sum_{p=1}^{n}\frac{1}{g'}\frac{\del g'}{\del y_p}\frac{u_{p}}{u}
\end{aligned}
\]
On the other hand, by~\eqref{eq: LinearizedEquation} we have
\[
u_1\sum_p \frac{u_{1pp}}{u_{pp}}  + \frac{1}{g'}\frac{\del g'}{\del y_1}u_{11}u_1 = 0
\]
and so, in total
\[
u_1\sum_p \frac{u_{1pp}}{u_{pp}}-\sum_{p=1}^{n}\frac{1}{g'u_{11}}\frac{\del g'}{\del y_p}u_{11p}=\sum_{p=1}^{n}\frac{1}{g'}\frac{\del g'}{\del y_p}\frac{u_{p}}{u}.
\]
Substituting this expression into~\eqref{eq: PogorelovInequality1} yields
\begin{equation}
\begin{aligned}
0 &\geq \sum_{2 \leq p,q\leq n}\frac{u_{1pq}^2} {u_{pp}u_{qq}u_{11}}-u_{11}\frac{\del^2 \log g'}{\del y_1^2}  \\
  &\quad + \frac{n}{u} -\frac{u_1^2}{u^2u_{11}} + u_{11}+\sum_{p=1}^{n}\frac{1}{g'}\frac{\del g'}{\del y_p}\frac{u_{p}}{u}.
  \end{aligned}
\end{equation}
Since $u<0$, the structural property $(iv)$ gives
\[
\frac{\del^2 \log g'}{\del y_1^2} \leq 0 \quad \text{ and } \quad \frac{1}{g'}\sum_p\frac{\del g'}{\del y_p}\frac{u_{p}}{u} \geq \frac{\kappa}{u}.
\]
 Thus, in total, at $x_*$ we have
 \[
 0 \geq \frac{n+\kappa}{u} -\frac{u_1^2}{u^2u_{11}} + u_{11}.
 \]
 Multiplying through the $|u|^2u_{11}$ yields the result.

 We now address the case when $x_* \in \del\Omega \cap S_h \cap B_1(0)$.  First we remark that since $g,g'$ are smooth and bounded strictly away from zero, the result of Chen-Liu-Wang \cite{Chen-Liu-Wang}, and the regularity theory of Caffarelli \cite{Caffarelli2, Caffarelli3}, $u$ is smooth up to the boundary 
 
 By obliqueness, we can apply a shearing transformation in the $(x_2,\ldots,x_{n})$ directions which fixes the $x_n$ direction, we may assume that $e_n$ is the inward pointing normal vector to $\Omega$ at $x_*$.  In other words, we have $\frac{\del \phi_{\Omega}}{\del x_j}(x_*)=0$ for all $1 \leq j \leq n-1$.  Similarly, by a shearing transformation in the $(y_1,\ldots,y_n)$ variables which fixes the $y_n$ direction we may assume that  $e_n$ is the inward pointing normal vector to $\Omega'$ at $y_*=\nabla u(x_*)$.  In other words, $\frac{\del \phi_{\Omega'}}{\del y_j}(y_*)=0$ for all $1 \leq j \leq n-1$. 
 
 We may write the restriction of  $\Phi$ to $\del\Omega \cap S_h \cap B_1(0)$ as
 \[
 \Phi\big|_{\del \Omega}(x_1,x_2,\ldots,x_{n-1}) = \Phi(x_1,\ldots,x_{n-1},\phi_\Omega(x_2,\ldots,x_{n-1}))
 \]
 Since $\Phi$ is maximized at $x_*$ we have
\begin{equation} \label{eq: boundaryFirstDerivTest}
 0=\frac{\del \Phi}{\del x_j} + \frac{\del \Phi}{\del x_n}\frac{\del \phi_{\Omega}}{\del x_j} = \frac{\del \Phi}{\del x_j} \quad \text{ for }1\leq j\leq n-1.
 \end{equation}
 Differentiating a second time yields
 \begin{equation}\label{eq: boundarySecondDerivTest}
 \begin{aligned}
0&\geq  \frac{\del^2 \Phi}{\del x_j^2} + \frac{\del \Phi}{\del x_n}\frac{\del^2 \phi_{\Omega}}{\del x_j^2} + 2\frac{\del^2 \Phi}{\del x_n\del x_{j}}\frac{\del \phi_{\Omega}}{\del x_j} + \frac{\del ^2\Phi}{\del x_n^2}\left(\frac{\del \phi_{\Omega}}{\del x_j}\right)^2\\
&= \frac{\del^2 \Phi}{\del x_j^2} + \frac{\del \Phi}{\del x_n}\frac{\del^2 \phi_{\Omega}}{\del x_j^2}\bigg|_{x_*}
\end{aligned}
 \end{equation}
We now consider the derivative of $\Phi$ in the $x_n$ direction.  We have
 \[
 \Phi_{n} = \frac{u_{11n}}{u_{11}}+\frac{u_n}{u} + u_1u_{1n}
 \]
 On the other hand, from the boundary data we have
 \[
 u_{n} = \phi_{\Omega'}(u_1,\ldots,u_{n-1})
 \]
 and so
 \[
 \begin{aligned}
 u_{1n} &= \sum_{p=1}^{n-1}\frac{\del \phi_{\Omega'}}{\del y_p}u_{1p} \\
  u_{1n}(x_*) &= \sum_{p=1}^{n-1}\frac{\del \phi_{\Omega'}}{\del y_p}(y_*)u_{1p}(x_*) = 0.
  \end{aligned}
 \]
 Differentiating again yields
 \[
 u_{11n} = \sum_{1 \leq p,q \leq n-1}\frac{\del^2 \phi_{\Omega'}}{\del y_p\del y_q}u_{1p}u_{1q} + \sum_{p=1}^{n-1}\frac{\del \phi_{\Omega'}}{\del y_p}u_{11p}.
 \]
By~\eqref{eq: boundaryFirstDerivTest} we have
\[
\Phi_p = \frac{u_{11p}}{u_{11}} + \frac{u_p}{u} + u_{1p}u_1=0 \quad \text{ for } 1 \leq p \leq n-1
\]
and so
\[
\frac{u_{11n}}{u_{11}} = \sum_{1 \leq p,q \leq n-1}\frac{\del^2 \phi_{\Omega'}}{\del y_p\del y_q}\frac{u_{1p}u_{1q}}{u_{11}} - \sum_{p=1}^{n-1}\frac{\del \phi_{\Omega'}}{\del y_p}\left(\frac{u_p}{u} + u_{1p}u_1\right).
\]
This gives
\[
\begin{aligned}
\Phi_n &= \sum_{1 \leq p,q \leq n-1}\frac{\del^2 \phi_{\Omega'}}{\del y_p\del y_q}\frac{u_{1p}u_{1q}}{u_{11}} - \sum_{p=1}^{n-1}\frac{\del \phi_{\Omega'}}{\del y_p}\left(\frac{u_p}{u} + u_{1p}u_1\right) + \frac{\phi_{\Omega'}}{u} + u_1 \sum_{p=1}^{n-1}\frac{\del \phi_{\Omega'}}{\del y_p}u_{1p}\\
&= \sum_{1 \leq p,q \leq n-1}\frac{\del^2 \phi_{\Omega'}}{\del y_p\del y_q}\frac{u_{1p}u_{1q}}{u_{11}} +\frac{1}{|u|} \left(\sum_{p=1}^{n-1}\frac{\del \phi_{\Omega'}}{\del y_p}u_p - \phi_{\Omega'}\right).
\end{aligned}
\]
By the convexity of $\phi_{\Omega}$, the first term is non-negative. By the convexity of $\phi_{\Omega'}$ and $\phi_{\Omega'}(0) = 0$, the second term is non-negative as well. Hence we conclude that
\[
\Phi_{n}(x_*) \geq 0. 
\]
In particular, we must have
\begin{equation}\label{eq: normalDerivBoundaryPog}
\Phi_{n}(x_*)=0 \quad \text{ and } \quad  \Phi_{nn}(x_*) \leq 0.
\end{equation}
Substituting this into~\eqref{eq: boundarySecondDerivTest} we obtain
\begin{equation}\label{eq: nonNormalDerivBoundaryPog}
\frac{\del \Phi}{\del x_j}(x_*)=0 \quad \text{ and } \quad \frac{\del^2 \Phi}{\del x_j^2}(x_*) \leq 0. 
\end{equation}
Finally, since $u_{1n}(x_*)=0$, we may perform a shearing transformation preserving the $x_1$ and $x_n$ directions so that $D^2u(x_*)$ is diagonal.  By structural condition $(i)$, $\del \Omega'$ is invariant under translations in the $x_1$ direction and so equations~\eqref{eq: LinearizedEquation},~\eqref{eq: concavityEquation} hold up to the boundary.  Furthermore, by~\eqref{eq: normalDerivBoundaryPog} and ~\eqref{eq: nonNormalDerivBoundaryPog} we conclude that ~\eqref{eq: firstDerivTestPhi} and~\eqref{eq: secondDerivTestPhi} also hold at $x_*$.  We may therefore treat $x_*$ as if it were an interior point, and the previous argument leads to the desired estimate.
\end{proof}

\begin{remark}
    In Proposition~\ref{prop: VeryGenPog}, the non-degeneracy of the measures $g(x)dx$ and $g'(y)dy$, as well as the smoothness of the boundaries $\Omega,\Omega'$ can typically be relaxed by an approximation argument; see for example \cite[Proposition 4.1]{Jhaveri-Savin}.  We given an example of such an argument in the proof of Theorem~\ref{thm: secMixedHomog} below.
\end{remark}

\subsection{Shape of sections in mixed homogeneity case}
We now prove the main result of this section, which gives the shape of sections when the domains has a flat side. As we can see, the flat sides breaks the homogeneity, and the solution exhibits mixed homogeneous behaviour. 

\begin{theorem}\label{thm: secMixedHomog}
    Suppose that near the origin $\Omega$ is locally equal to $\mathbb R^{n-m}\times \mathtt C$ and $\Omega'$ is locally equal to $\mathbb R^{n-1}\times\mathbb R_+$, and $g(x) =1, g'(y) = y_n^k$. Then there exist $C>1$ such that for $h\ll1$, we have
    \[B_{C^{-1}h^{\frac 1 2}}^{n-m}\times B_{C^{-1}h^{\frac{1}{1+\frac{m}{m+k}}}}^m\subset S_h^c(u, 0)\subset B_{Ch^{\frac 1 2}}^{n-m}\times B_{Ch^{\frac{1}{1+\frac{m}{m+k}}}}^m,\]
    and
    \[B_{C^{-1}h^{\frac 1 2}}^{n-m}\times B_{C^{-1}h^{\frac{1}{1+\frac{m+k}{m}}}}^m\subset S_h^c(v, 0)\subset B_{Ch^{\frac 1 2}}^{n-m}\times B_{Ch^{\frac{1}{1+\frac{m+k}{m}}}}^m.\]
\end{theorem}

\begin{proof}

Let $S_h^c(v, 0)\sim E(l_1, l_2, \ldots, l_{n-1}, d(h))$ and $S_h^c(u, 0)\sim E(\frac{h}{l_1}, \frac{h}{l_2}, \ldots, \frac{h}{l_{n-1}}, \frac{h}{d(h)})$ and assume without loss of generality that
\[l_1\leq l_2\leq\cdots\leq l_{n-1}.\]  First we claim that by the Pogorelov estimate (Proposition~\ref{prop: VeryGenPog}) we have 
\begin{equation}\label{eq: secShapeFlatDirections}
l_1\sim l_{n-m}\sim h^{\frac{1}{2}}. 
\end{equation}
We first prove that $\ell_{n-m} \lesssim h^{\frac{1}{2}}$ using Proposition~\ref{prop: VeryGenPog} and an approximation argument.  By the duality of sections, and rotational symmetry in the $(x_1,\ldots, x_{n-m})$ variables, it suffices to prove the bound
\[
\sup_{S_{h_0}(u,0)} u_{11} \leq C
\]
Suppose that $u$ satisfies the assumptions of the theorem, and let $v=u^*$ be the Legendre dual of $u$.  Fix some $h_0$ and some $\delta >0$ so that $S_{h_0}(u,0) \subset B_{1-2\delta}$.  We require an approximation result similar to Lemma~\ref{lem: existApprox}; since the details are different we briefly sketch the construction.  Choose a sequence $\mathtt{C}_j \subset \mathbb{R}^{m}$ of smooth convex sets approximating $\mathtt{C}$ in $B_1(0)$, and choose also a sequence of sets $\Upsilon_j\subset \{y \in \mathbb{R}^n: y_n \geq 0\}$  which are dilations of  $\Upsilon:= \nabla u((\mathbb{R}^k\times \mathtt{C})\cap B_1(0))$, such that for $j$ sufficiently large
\begin{itemize}
\item we have the containment
\[
\Upsilon_{\delta}:=\nabla u(((\mathbb{R}^{n-m}\times \mathtt{C}) \cap B_{1-\delta}(0)) \Subset \Upsilon_{j}. 
\]

\item the mass balancing condition holds:
\[
\int_{(\mathbb{R}^{n-m}\times \mathtt{C}_j)\cap B_1(0)} dx = \int_{\Upsilon_j} (y_n+\frac{1}{j})^{k} dy
\]
\end{itemize}
Let $v_j$ be a convex function solving the optimal transport problem
\begin{equation}\label{eq: PogOTEqnApprox}
    \begin{aligned}
    \det D^2v_j&=(y_n+j^{-1})^{k}, \quad \text{ in  }\, \Upsilon_j \\
    \nabla v_j(\Upsilon_j) &= (\mathbb{R}^{n-m}\times \mathtt{C}_j)\cap B_1(0)
    \end{aligned}
    \end{equation}

By the regularity theory for optimal transport maps \cite{Caffarelli, Caffarelli2, Chen-Liu-Wang, Jhaveri-Savin} (or by Theorem~\ref{thm: global-C-2-a-regularity}) we have that, for any $j$, $v_j \in C^{2,\alpha}(\overline{\Upsilon}_{\delta})$ and the $v_j$ are uniformly bounded in $C^{1,\alpha}(\overline{\Upsilon}_{\delta})$, and hence $v_j$ converge to $v$ in $C^{1,\alpha}(\overline{\Upsilon}_{\delta})$.  In particular, if $u_j$ denotes the Legendre transform of $v_j$, then for $j$ large $u_j$ satisfies 
\begin{equation}\label{eq: PogOTEqn}
    \begin{aligned}
    (\frac{\del u_j}{\del x_n}+j^{-1})^{k}\det D^2u&=1, \quad \text{ in  }\, (\mathbb{R}^{n-m} \times \mathtt C_j)\cap S_{h_0}(u_j,0) \\
    \frac{\del u_{j}}{\del x_n} &\geq 0, \quad \text{ in  }\, (\mathbb{R}^{n-m} \times \mathtt C_j)\cap S_{h_0}(u_j,0)\\
    \frac{\del u_{j}}{\del x_n}&=0,\quad  \text{ on } (\mathbb{R}^{n-m} \times \del \mathtt C_j) \cap S_{h_0}(u_j,0)
    \end{aligned}
    \end{equation}
We may therefore apply Proposition~\ref{prop: VeryGenPog} uniformly to $u_j$ with $\kappa =k$, and the result follows by taking the limit as $j\rightarrow \infty$.

Next we prove the reverse estimate $\ell_1 \gtrsim h^{\frac{1}{2}}$, which will follow from the estimate 
\[
\sup_{S_{h_0}(v,0)} v_{11} \leq C
\]
for some $C, h_0$. Since the approximation argument is similar to the preceding argument, we only sketch the details. 
We choose a sequence of smooth convex functions $\phi_{\mathtt{C}_m}$ such that $ \phi_{\mathtt{C}_m}(0)= 0$, and $\phi_{\mathtt{C}_m}$ converge to $\phi_{\mathtt{C}}$ uniformly on compact sets.  Let
    \[
    \mathtt{C}_m = \{ x_n \geq \phi_{\mathtt{C}_m}(x_1,\ldots, x_{n-1})\}
    \]
    We now solve a sequence of optimal transport problems with smooth non-degenerate measures given by $dx$ and $(y_n+\epsilon)^kdy$, and targets contained in $\mathtt{C}_m$.  The result follows by applying the preceding Pogorelov interior estimate, and passing to the limit as $m\rightarrow \infty$.

We have thus established~\eqref{eq: secShapeFlatDirections}. By similar arguments to those in the proof of Theorem~\ref{thm: round-degenerate-density} we will show that
\[
d(h)\lesssim l_{n-m+1}\leq \cdots \leq l_{n-1}. \]
Indeed, let $\{e_1,\ldots, e_n\}$ be the standard basis of $\mathbb{R}^n$ and let $S_{h}^{c,i}(v,0) = S_{h}^{c}(v,0) \cap {\rm Span}\{e_i,e_n\}$.  Define
\[
\begin{aligned}
-\ell_i(h) &= \min \{x_i : (x_i,0)\in S_{h}^{c,i}(v,0)\}\\
r_i(h) &=\max \{x_i : (x_i,0)\in S_{h}^{c,i}(v,0)\}
\end{aligned}
\]
and note that $\ell_i(h) \sim r_i(h)$ by balancing. Since $\mathtt{C}$ is a strict cone, we have
\[
\phi_{\mathtt{C}}(x_{n-m+1}, \ldots, x_{n-1}) \geq c \max_{n-m+1\leq i \leq n-1} |x_i|
\]
for some uniform $c>0$ depending only on $\mathtt{C}$.  For any $n-m+1 \leq i \leq n-1$, the boundary values yield
\[
\begin{aligned}
\frac{\del v}{\del y_n}(-\ell_i(h)e_i) &= \phi_{\mathtt{C}}(\frac{\del v}{\del y_{n-m+1}}(-\ell_i(h)e_i), \ldots, \frac{\del v}{\del y_{n-1}}(-\ell_i(h)e_i)) \\
&\geq c|\frac{\del v}{\del y_i}(-\ell_i(h)e_i)|\\
& \geq c\frac{\ell_i(h)}{h}
\end{aligned}
\]
where we used the convexity of $v$ in the last line.  From now on we consider the two dimensional subspace $\mathbb{R}^{2} = {\rm Span}\{e_i, e_n\} \subset \mathbb{R}^{n}$ for $n-m+1 \leq i \leq n$.  Let
\[
\vec{\eta} = \left(\frac{\del v}{\del y_i}(-\ell_i(h)e_i),\frac{\del v}{\del y_n}(-\ell_i(h)e_i)\right).
\]
Then we have shown that $\vec{\eta}\cdot e_n >c |\vec{\eta} \cdot e_i|$, and the line
\[
(y_i+\ell_i(h), y_n)\cdot \eta =0
\]
is a supporting line for $S_{h}^{c,i}(v,0)$,intersecting the line $\{r_i(h)e_i+te_n: t\in \mathbb{R}\}$ at a point $r_i(h)e_i+t_*e_n$ where $t_* \sim \ell_i(h)$.  Therefore,
\[
d_h \lesssim \min_{n-m+1 \leq i \leq n-1}\{\ell_i(h), r_i(h)\} \lesssim l_{n-m+1}
\]
as desired.

From the mass balancing condition imposed by the equation, we have 
\[\frac{h^n}{l_1\dots l_{n-1}d(h)}\sim |S_h(u, 0)|\sim \int_{S_h(v, 0)}y_n^kdy \sim l_1\cdots l_{n-1}d(h)^{k+1}\]
which implies
\[h^m\sim (l_{n-m+1}\cdots l_{n-1})^2d(h)^{k+2} \gtrsim d(h)^{k+2+2(m-1)} \]
and so
\[
|S_{h}(u,0)| \lesssim h^{\frac{n-m}{2}+\frac{m}{1+\frac{m}{m+k}}}.
\]
We claim that a comparable lower volume bound holds 
\begin{lem}\label{lem: secVolLowBndMixed}
In the above setting, there is a constant $c>0$ so that
\begin{equation}\label{eq: SecVolMixHomog}
|S_h(u, 0)|\geq c h^{\frac{n-m}{2}+\frac{m}{1+\frac{m}{m+k}}}
\end{equation}
for all $h \leq 1$.
\end{lem}

\noindent Assuming Lemma~\ref{lem: secVolLowBndMixed}, the shape of the sections follows as in Theorem~\ref{thm: round-degenerate-density}.
\end{proof}

It only remains to prove Lemma~\ref{lem: secVolLowBndMixed}.

\begin{proof}[Proof of Lemma~\ref{lem: secVolLowBndMixed}]
We argue by induction on $m$.  First note that it suffices to prove~\eqref{eq: SecVolMixHomog} for all $h$ sufficiently small, up to redefining $c$.  When $m=1$, the result follows from the work of Jhaveri-Savin \cite[Theorem 1.4]{Jhaveri-Savin}. Suppose that~\eqref{eq: SecVolMixHomog} is not true.  Set
\[
\gamma = \frac{n-m}{2}+\frac{m}{1+\frac{m}{m+k}}
\]
For $\epsilon>0$ define the set of ``$\epsilon$-bad scales" to be
\[
\mathfrak{B}(\epsilon) = \left\{h \in \mathbb{R}_+ : 0<h \leq 1 \, \text{ and } \,\epsilon h^{\gamma} \leq |S_{h}(u,0)| \leq 2\epsilon h^{\gamma}\right\}
\]
and let
\[
\widehat{\mathfrak{B}}(\epsilon) := \bigcup_{\tau \leq \epsilon} \mathfrak{B}(\tau)= \left\{h \in \mathbb{R}_+ : 0<h \leq 1 \, \text{ and } \, |S_{h}(u,0)| \leq 2\epsilon h^{\gamma}\right\}
\]
Suppose that, for some $\epsilon >0$ we have $\inf \widehat{\mathfrak{B}}(\epsilon) = \delta >0$.  Then, for all $h<\delta$ we have
\[
|S_{h}(u,0)| \geq 2\epsilon h^{\gamma}
\]
which is the desired result.  Thus we may assume that there is a sequence $h_i \searrow 0$, and a sequence $\epsilon_i\searrow 0$ such that $h_i \in \mathfrak{B}(\epsilon_i)$.  We will show that this scenario leads to a contradiction.  The main technical tool is the following claim:

\begin{claim}\label{claim: noCollapsing}
Fix a constant $K\geq 1$.  There exists $h_0, \epsilon_0$, and $\lambda_0 \in (0,1)$ such that, if $h < h_0$, and $\epsilon < \epsilon_0$, and $\epsilon h^{\gamma} \leq |S_h(u,0)| \leq 2\epsilon h^{\gamma}$, then for some $h' \in (\lambda_0 h, h)$ we have $|S_{h'}(u,0)| \geq 2K\epsilon (h')^{\gamma}$.
\end{claim}

We now prove~\eqref{eq: SecVolMixHomog} and hence Lemma~\ref{lem: secVolLowBndMixed}, assuming Claim~\ref{claim: noCollapsing}.  First, observe that by convexity, we have the containment
\[
\frac{1}{2}S_{h}(u,0) \subset S_{\frac{h}{2}}(u,0)
\]
and hence, if $h \in \mathfrak{B}(\epsilon)$
\[
|S_{\frac{h}{2}}(u,0)| \geq 2^{-n}|S_{h}| \geq 2^{-n+\gamma} \epsilon \left(\frac{h}{2}\right)^{\gamma}
\]
In particular, we have
\begin{equation}\label{eq: halfScaleControlDegen}
h \in \mathfrak{B}(\epsilon) \Rightarrow \frac{h}{2} \notin \widehat{\mathfrak{B}}(2^{-n+\gamma} \epsilon).
\end{equation}
We apply the claim with $K = 2^{n-\gamma}>1$.  Let $\epsilon_0, h_0, \lambda_0$ be as in the claim.  Let 
\[
h_1 = \inf \{ h< h_0 : h \notin \widehat{\mathfrak{B}} (\epsilon_0)\}.
\]
By~\eqref{eq: halfScaleControlDegen} we have $h_1 \in \mathfrak{B}(\epsilon_1)$ for some $\epsilon_1 \in (K^{-1} \epsilon_0, \epsilon_0)$.  By Claim~\ref{claim: noCollapsing} we conclude that there is an $h_2\in (\lambda_0h_1,h_1)$ such that $h_2 \notin \widehat{\mathfrak{B}}(K\epsilon_1)$.  In particular, we obtain
\[
h_2 \in \mathfrak{B}(\epsilon_2) \quad \text{ for } \epsilon_2 >K\epsilon_1 > \epsilon_0
\]
Thus, for every $h\in(h_2,h_1) \subset (\lambda_0h_1,h_1)$ we have
\[
|S_{h}(u,0)| \geq |S_{h_2}(u,0)| \geq \epsilon_0h_2^{\gamma} \geq \epsilon_0\lambda_0^{\gamma}h_1^{\gamma} \geq \epsilon_0\lambda_0^{\gamma}h^{\gamma}
\]
We can now repeat the argument, taking 
\[
h_3 = \inf \{ h< h_2 : h \notin \widehat{\mathfrak{B}} (\epsilon_0)\},
\]
and proceeding as before.  It follows that
\[
|S_{h}(u,0)| \geq \epsilon_0\lambda_0^{\gamma}h^{\gamma} \quad \text{ for all } h < h_0
\]
which contradicts our assumption regarding the existence of the sequence $(h_i, \epsilon_i)$.
\end{proof}

In order to complete the proof of Lemma~\ref{lem: secVolLowBndMixed} (and thereby Theorem~\ref{thm: secMixedHomog}), it only remains to prove Claim~\ref{claim: noCollapsing}.  

\begin{proof}[Proof of Claim~\ref{claim: noCollapsing}]
The proof is by contradiction. If the claim is not true, then there exist $h_p, \epsilon_p, r_p, \delta_p, \mu_{p} \in (0,1)$ such that $h_p \searrow 0$, $\epsilon_p \searrow 0$, $\mu_p \searrow 0$ as $p\to \infty$, and
\[
r_p<h_{p},  \quad \delta_p <\epsilon_{p}, \quad r_p \in \mathfrak{B}(\delta_p)
\]
and such that, for all $h' \in(\mu_pr_p, r_{p})$ we have $h' \in \widehat{\mathfrak{B}}(K\delta_p)$.  By assumption we have $r_p \rightarrow 0$ and 
\[
\frac{|S_{r_p}(u,0)|}{r_p^{\gamma}} \leq 2\delta_p \rightarrow 0.
\]
From the preceding analysis of the sections, this can only occur if  $\frac{d(r_p)}{l_{n-1}(r_p)} \rightarrow 0$.  Hence, the rescaled functions 
\[
u_p(x) := \frac{u(\frac{r_{p}}{l_1(r_p)}x_1, \ldots, \frac{r_p}{l_{n-1}(r_p)}x_{n-1}, \frac{r_p}{d(r_p)}x_n)}{r_p}
\]
satisfy
\[\det D^2u_p = \frac{r_p^{n}}{(l_1(r_p)\cdots l_{n-1}(r_p))^2d(r_p)^{k+2}}\frac{1}{(u_p)_n^k}\]
and since $\frac{h^{n}}{(l_1\cdots l_{n-1})^2d(h)^{k+2}}\sim 1$, we can take a convergent subsequence converging to $u_{\infty}$, which is a limiting optimal transport map $((\mathbb R^{n-m+j}\times \tilde{\mathtt C}, dx), (\mathbb R^{n-1}\times \mathbb R_+, y_n^kdy), u_{\infty}, v_{\infty})$ for some $j>0$ and some strict cone $\tilde{\mathtt C}\subset \mathbb R^{m-1-j}\times \mathbb R_+$. Therefore, by the induction hypothesis, there are constants $c, C$  so that for all $h\leq 1$ there holds 
\[
ch^{\frac{n-m+p}{2}+\frac{m-p}{1+\frac{m-p}{m-p+k}}}\leq |S_h(u_{\infty}, 0)|\leq Ch^{\frac{n-m+p}{2}+\frac{m-p}{1+\frac{m-p}{m-p+k}}}.
\]

If we denote $\beta:=\frac{n-m+j}{2}+\frac{m-j}{1+\frac{m-j}{m-j+k}}$, then since $j>0$, we have $\beta < \gamma$.  Fix $\Lambda$ large so that
\[
\frac{1}{2}\frac{c}{C}\Lambda^{\gamma-\beta}>4K.
\]
Then, for $p$ sufficiently large we have
\[
\begin{aligned}
|S_{\frac{r_p}{\Lambda}}(u, 0)|\geq \frac{c}{2C}\frac{1}{\Lambda^{\beta}}|S_{r_p}(u, 0)| &= \Lambda^{\gamma-\beta}\frac{c}{2C}\frac{1}{\Lambda^{\gamma}}|S_{r_p}(u,0)|\\
& \geq 4K\frac{1}{\Lambda^{\gamma}}|S_{r_p}(u,0)|\\
& \geq 4K\frac{1}{\Lambda^{\gamma}}\delta_pr_p^{\gamma}
\end{aligned}
\]
and so $\frac{r_p}{\Lambda} \notin \widehat{B}(K\delta_p)$ for $p$ large, a contradiction.
\end{proof}

\end{document}